\documentclass[oneside]{amsart}
\usepackage{amsmath, amssymb,color,amsthm,graphicx}
\usepackage[dvipsnames]{xcolor}
\usepackage[backref]{hyperref}
\theoremstyle{definition}
\newtheorem{theorem}{Theorem}[section]
\newtheorem{prop}[theorem]{Proposition}
\newtheorem{lemma}[theorem]{Lemma}
\newtheorem{corollary}[theorem]{Corollary}
\newtheorem{rmk}[theorem]{Remark}

\theoremstyle{definition}

\newtheorem{main}{Theorem}

\newtheorem{definition}[theorem]{Definition}
\usepackage{graphicx}
\usepackage{url}
\graphicspath{{../../img/}{../img/}{./img/}}
\numberwithin{equation}{section}
\usepackage[hang,small,bf]{caption}
\usepackage[subrefformat=parens]{subcaption}
\small\normalsize
\def\R{\mathbb{R} }

\def\RR{\mathbb{R}}
\def\ZZ{\mathbb{Z}}

\def\vphi{\varphi}
\def\del{\partial}
\def\udel{\overline{\del}}
\def\ldel{\underline{\del}}

\def\vol{\text{vol}}

\DeclareMathOperator{\Int}{Int}

\DeclareMathOperator{\sgn}{sgn}
\DeclareMathOperator{\ra}{{\rightarrow}}

\date{\today}

\subjclass[2010]{}

\title[]{Length averages for codimension one foliations}

\author[Asaoka]{Masayuki Asaoka}
\address[Masayuki Asaoka]{Faculty of Science and Engineering, Doshisha University, Kyoto, 610-0394, JAPAN}
\email{masaoka@mail.doshisha.ac.jp}

\author[Nakano]{Yushi Nakano}
\address[Yushi Nakano]{Graduate School of Science, Hokkaido University, Hokkaido, 060-0810, Japan}
\email{yushi.nakano@math.sci.hokudai.ac.jp}

\author[Varandas]{Paulo Varandas}
\address[Paulo Varandas]{CMUP, Faculdade de Ci\^encias da Universidade do Porto \& Departamento de Matem\'atica e 
Estat\'istica, Universidade Federal da Bahia, BRAZIL}
\email{paulo.varandas@ufba.br}

\author[Yokoyama]{Tomoo Yokoyama}
\address[Tomoo Yokoyama]{Graduate School of Science and Engineering, Saitama University, Saitama, 338-8570, JAPAN}
\email{tyokoyama@rimath.saitama-u.ac.jp}


\begin{document}

\begin{abstract}
In this paper we study geometrical and dynamical properties of codimension one foliations, by exploring a relation between length averages and ball averages of certain group actions.
We introduce a new mechanism, which relies on the group structure itself, to obtain irregular behavior of ball averages for certain non-amenable group actions. 
Several geometric realization results show that any such groups can appear connected with the topology of leaves which are connected sums of plugs with a special geometry, namely nearly equidistant boundary components. 
This is used to produce the first examples of codimension one $\mathcal C^\infty$ regular foliations on a compact Riemannian manifold $M$ for which the length average of some continuous function does not exist on a non-empty open subset of $M$. 
\end{abstract}

\maketitle
\section{Introduction}\label{intro}

The study of foliations has attracted the attention of many mathematicians in the past decades due to the richness of their dynamical features and connections with several areas in mathematics. Several approaches have been taken to fully understand foliations, which use their geometrical, topological, dynamical, and algebraic aspects.
Indeed, a smooth regular foliation on a compact Riemannian manifold has a finitely generated holonomy pseudo-group $G$ (determined by return maps to local transversal sections and unique up to isomorphism) which is in parallel to Poincar\'e maps of a smooth flow. 
In particular, this transversal dynamics suggested the study of recurrence for the orbits of points under holonomy pseudo-groups and led to the concept of resilient leaves (see e.g. \cite{GLW,sacksteder1965foliations}). 
Thus, it is natural to des\-cribe the relation between the topology of the manifold, the algebraic properties of such holonomy pseudo-group and geometric properties of the leaves of the foliation, which are naturally endowed with a complete Riemannian metric.
This program has been pursued since the fifties, inspired by classification problems and questions concerning the growth of volume on leaves of the foliation,  and much is known especially in the codimension one setting. 
For instance, for a $\mathcal C^2$ codimension one foliation on a compact manifold, the existence of a resilient leaf is equivalent to the positivity of the geometric entropy of the foliation (cf.~\cite{GLW}), and ensures the existence of a leaf in which the volume of balls grows exponentially fast with the radius \cite{Du}. 
Yet, the general picture seems incomplete and there are several open problems and conjectures. We refer the reader to  \cite{candel2000foliations,candel2000foliationsb,Hu1,Hu2, Wal} and references therein for very complete expositions on foliation theory.

In this paper, we aim to bridge some ideas from ergodic theory and dynamical systems to foliation theory to characterize the complexity of foliations. 
In opposition to the context of classical dynamical systems, where a one-dimensional foliation is generated by the orbits of a smooth flow, there are two sources of complexity in the dynamics of foliations. 
The first arises from the fact that the transversal dynamics is originated from a holonomy pseudo-group, which often fails to preserve invariant measures.
The second is due to the fact that there exists no time parameterization of the leaves of the foliation, which leads to parameterization of them by the Riemannian metric.
Related to the second, recently the concept of the length average of foliations was introduced in \cite{nakano2019existence} to investigate the complexity of foliations from the viewpoint of ergodic theory. 
Given a singular foliation $\mathcal F$ on a compact Riemannian manifold $M$, 
 the \emph{length average} of a continuous function $\varphi :M\to \mathbb R$ at a non-singular point $x\in M$ for $\mathcal F$ is given by
\begin{equation}\label{eq:maindef}
\lim _{r\to\infty} \frac{1}{\vert B_r^{\mathcal F}(x) \vert } \int _{B_r^{\mathcal F}(x)} \varphi (y) dy,
\end{equation}
where $B_r^{\mathcal F}(x)$ is the ball of radius $r>0$ at $x$ on the leaf $\mathcal F(x)$ of the foliation passing through $x$ and $\vert B_r^{\mathcal F}(x) \vert $ is the volume of $B_r^{\mathcal F}(x)$ with respect to the induced metric on $\mathcal F(x)$. 
This is a natural generalization of the time average of $\varphi$ at $x$ by a flow, that depends on the Riemannian metric, and can be thought as a refinement of the problem of finding leaves with exponential volume growth considered in \cite{Du}.

The non-existence of length averages (called \emph{historic behavior} or \emph{irregular behavior}) is known to be rare for smooth foliations on surfaces. 
Indeed, given a compact surface, length averages exist everywhere for all continuous functions for `almost all'   codimension one $\mathcal C^1$ non-degenerate singular foliations, all $\mathcal C^r$ non-degenerate singular foliations with $r>1$ and all $\mathcal C^r$ regular foliations with $r\ge 1$ (see \cite{nakano2019existence} for the precise statements). 
The underlying mechanism is that the one-dimensionality of the leaves of the foliation allows one to relate the length averages ~\eqref{eq:maindef} with C\`esaro averages of an integrated observable, a fact that does not occur for foliations whose leaves have dimension larger than one.
In particular, the following questions appear naturally:
\begin{itemize}
\item[$\diamond$] Do length averages exist everywhere for any codimension one singular foliation on a compact Riemannian manifold?
\smallskip
\item[$\diamond$] Do length averages exist on a Baire residual set for any codimension one foliation on a compact Riemannian manifold?
\smallskip
\item[$\diamond$] Do length averages exist Lebesgue everywhere for any codimension one foliation on a compact Riemannian manifold?
\end{itemize}
\smallskip

We remark that if we remove the compactness, the non-degeneracy or the codimension one property, every question can be answered negatively in a strong way (cf.~\cite[Problem 3]{nakano2019existence} for a more detailed discussion). 
In dynamical systems, the non-existence of time averages is known to arise mostly from complicated structures of the space of ergodic measures. 
 However, due to the possible non-existence of invariant measures for the holonomy pseudo-group, one may expect the non-existence of length averages to occur as a consequence of rich 
geometric and algebraic structures of foliations.
The purpose of the present paper is to provide partial answers to the previous questions and, as a byproduct,  to shed some light on a new mechanism to produce irregular behavior for group actions.
Indeed, given a surface obtained by gluing pants with nearly equidistant boundaries as in Figure \ref{fig:1} below, one can write the length averages of a special class of continuous observables as a ball average 
$$
\frac{1}{\vert G_n\vert } \sum _{a\in G_n} \psi (f_a(x))
$$
associated to a continuous observable $\psi$ and the group $G$ generated by the finite number of gluing maps on the circle, where $G_n$ denotes the ball of radius $n$ centered at the identity in $G$ (see Section~\ref{sec:main} for precise definition).
This allows us to build a bridge between length averages for foliations and ball (or spherical) averages for group actions, whose existence was considered in the context of pointwise ergodic theorem by Nevo-Stein, Lindenstrauss and  Bufetov among others (see \cite{nevo1994generalization,lindenstrauss2001pointwise,bufetov2002convergence} and references therein).

 \begin{figure}[htb]
\begin{center}
\includegraphics[scale=0.15]{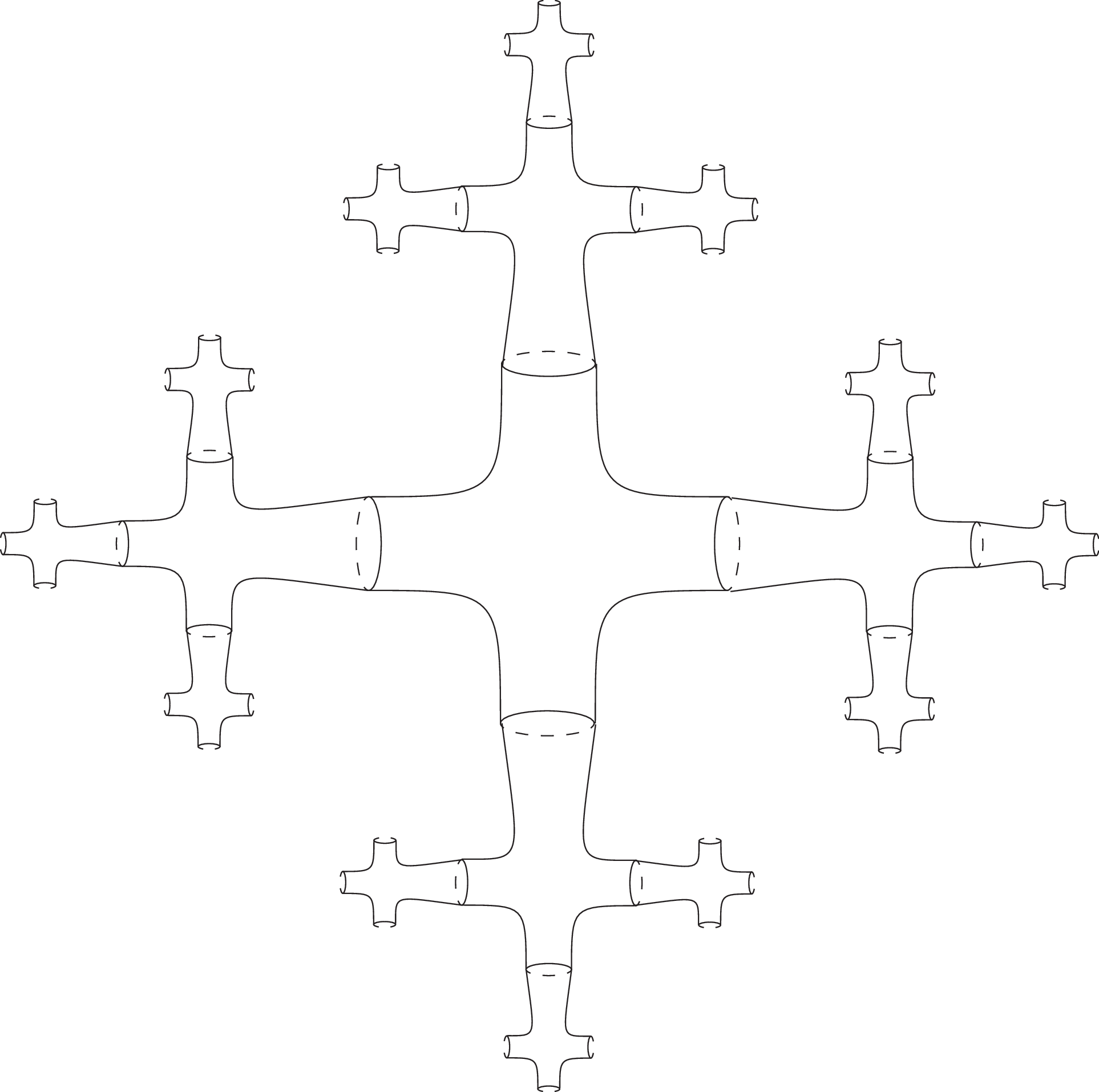}\includegraphics[scale=0.5]{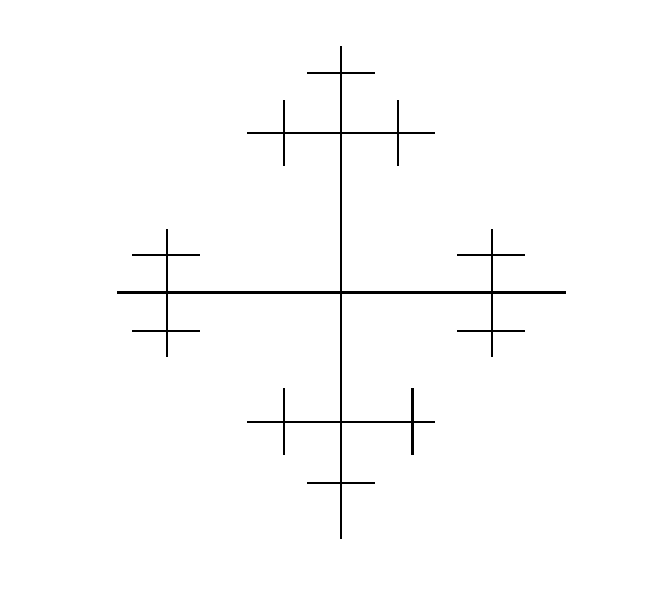}
\caption{A leaf of the foliation formed as a suspension of pants with nearly equidistant boundaries over the group action by the free group $F_2$ (on the left); Cayley graph of $F_2$ (on the right) }
\label{fig:1}
\end{center}
\end{figure}

If, in addition, the group $G$ is non-amenable or has a large boundary property (to be defined later), one can construct an open set (in particular, a positive Lebesgue measure set)  of points for which length averages do not exist. 
However, due to the uniqueness of geodesic rays, there exist no smooth pants with equidistant boundaries (cf.~Proposition~\ref{prop:2.2}).
Notwithstanding, we prove that the previous scheme does work for surfaces with corners and equidistant boundaries (Theorems~\ref{prop:pieces} and item (2) in Theorem~\ref{mainthm:4}) and $\mathcal C^\infty$ foliated manifolds where the leaves have nearly equidistant boundaries (Theorem~\ref{mainthm:2}).
Alternatively, if the group $G$ is amenable or has a small boundary property then length averages exist if and only if the ball averages of integrated observables exist as well (item (1) in Theorem~\ref{mainthm:4}).
Moreover, using suspensions of interval exchange maps, one can build an example of a $\mathcal C^1$ singular foliation on a surface such that the transversal dynamics has two ergodic invariant measures with dense basins of attraction, a condition which ensures that Baire typical orbits have irregular behavior (Theorem~\ref{mainthm:1}).

\section{Main results}\label{sec:main}

\subsection{Irregular behavior of length averages}
In this section, we state our main results, which provide answers to the main questions raised in the introduction. 
 Our first result is the following negative answer to the first and second questions of the introduction in the two-dimensional case.
 Recall that a subset of a topological space $U$ is said to be residual if it contains a countable intersection of open and dense subsets of $U$.
 Refer also to \cite{Stefan74} for the definition of non-degenerate singular foliations.

\begin{main}\label{mainthm:1}
There are a codimension one $\mathcal C^\infty$ non-degenerate singular foliation $\mathcal F$ on a compact surface $M$ and a continuous function on $M$ whose length averages for $\mathcal F$ do not exist on a residual subset of $M$.
\end{main} 

This complements  \cite{nakano2019existence}, where it was shown that for any codimension one $\mathcal C^r$ non-degenerate singular foliation on a compact surface with $r\ge 1$, the length average of any continuous function exists Lebesgue almost everywhere. 
In other words, the residual set in Theorem \ref{mainthm:1} is of zero Lebesgue measure, and the answer to the third question of the introduction in the two-dimensional case is positive. 
 In strong contrast to it, when $M$ is a 3-manifold, the answers to all the questions of the introduction are negative even in  the regular  foliations category:
   
\begin{main}\label{mainthm:2}
There are a codimension one $\mathcal C^\infty$ regular foliation on a compact three-dimensional manifold $M$ and a continuous function on $M$ whose length averages do not exist on a non-empty open set of $M$.
 \end{main} 
  
The proof of Theorem \ref{mainthm:2} ensures that there is
a $\mathcal C^0$ open set of continuous functions without length average on the open set of Theorem \ref{mainthm:2}.
In fact, 
it will be seen
that there are constants $\overline L,\underline L$ such that $\underline\ell\varphi(x)\le \underline L<\overline L \le \overline\ell\varphi(x)$ on the open set for the continuous function $\varphi$ of Theorem \ref{mainthm:2}, where $\overline\ell\varphi(x)$ and $\underline\ell\varphi(x)$ are defined by \eqref{eq:maindef} with ``$\limsup$'' and ``$\liminf$'' instead of ``$\lim$'' respectively.
Hence, the stronger statement holds due to 
the general observation that 
$\overline\ell\psi(x)-\underline\ell\psi(x)> \frac{\overline\ell\varphi(x)-\underline\ell\varphi(x)}{3}$ for any continuous function $\psi$ whose $\mathcal C^0$ distance to $\varphi$ is less than $\frac{\overline\ell\varphi(x)-\underline\ell\varphi(x)}{3}$.

By extending the foliation in Theorem \ref{mainthm:2}, one can also get a foliation on a compact manifold $M$ with dimension larger than three for which a length average does not exist on an open set of $M$.
  \begin{corollary}\label{cor:y}
 For any integers $d\ge 3$ and $1\le l \le d-2$, there are a codimension $l$ $\mathcal C^\infty$ regular foliation on a compact $d$-dimensional manifold $M$  and a continuous function on $M$ whose length averages do not exist on a non-empty open set of $M$.
 \end{corollary} 
 
 A codimension $l$ foliation without length averages with $l\ge 2$ was already given in \cite[Example 20]{nakano2019existence},  but the construction here is different from it. 
In fact, the foliation in \cite{nakano2019existence} is a suspension flow of a diffeomorphism without time averages, so that the dimension of the foliation is one, while the dimension of the foliation in Corollary \ref{cor:y} is larger one.
The reduction procedure that we employ to prove  Corollary \ref{cor:y} from Theorem \ref{mainthm:2}  can apply to much more general foliations  (see Proposition \ref{prop:Tk}). 

\subsection{Mechanism: length averages and ball averages}
To systematically understand the mechanism of the existence and non-existence of length averages on a positive Lebesgue measure set,\footnote{Since foliated manifolds with corners naturally appear in natural science, such as Moffatt eddies in fluid mechanics (cf.~\cite{Shankar2005, BBTB2011}), this extended question and results in this subsection would be appealing in itself.} 
including the new mechanism of irregular behavior in Theorem \ref{mainthm:2}, we extend the questions of the introduction to a \emph{smooth manifolds with corners} 
(i.e.~a topological space locally modeled on $[0,\infty )^l \times \mathbb R^{d-l}$ with some $2\le l\le d$, as a smooth manifold is locally modeled on $\mathbb R^d$, see e.g.~\cite{Joyce2012} for its precise definition).
 Notice that even in the case $(M,\mathcal F)$ is a foliated smooth Riemannian manifold with corners, the length average in \eqref{eq:maindef} of any continuous function at any non-singular point of $\mathcal F$ is well-defined. 
 We will investigate the relation between the existence of the length averages for the suspension 
 of a group action and a F{\o}lner type property (a small boundary property)
  of the group action.
 This would also explain the idea of the construction of the smooth foliated manifold in Theorem \ref{mainthm:2}.

 Let $G$ be a group defined by generators $G_1= \{ \alpha_1^\pm , \alpha_2^\pm , \ldots , \alpha_k^\pm \}$  (possibly with relations  besides $\alpha _i\alpha _i^{-1} = \alpha _i^{-1}\alpha _i=1$), and let $f$ be a $G$-action on a topological space $Y$ such that $ f(a):Y\to Y$ is a homeomorphism for all $a\in G$. 
  Throughout this paper, we identify an element $a$ in $G$ and the corresponding homeomorphism $f(a)$. 
     
 Let $G_n$ be the ball of radius $n$ at $1$ in $G$ (i.e.~the subset of $G$ consisting of reduced and non-empty words $a=a_1a_2 \cdots a_m$ with $1\leq m\leq n$ and $a_i\in G_1$).
Given a continuous function $\varphi$ on $Y$, the \emph{ball averages}  $\beta  _n\varphi$ of $\varphi$ by $G$ are defined as 
\[
\beta _n\varphi (x) = \frac{1}{\vert G_n\vert } \sum _{a\in G_n} \varphi (f_a(x)),
\]
where we simply write $f_a$ for $f(a)$.
See e.g.~\cite{nevo1994generalization,lindenstrauss2001pointwise,bufetov2002convergence,CCSV2021} and references therein for the study on the (non-)existence of ball averages (although their approaches are based on the existence of invariant measures while ours are not).

In what follows, we additionally assume that $f=(f_a)_{a\in G}$ is a finitely generated group action by diffeomorphisms on a smooth Riemannian manifold $Y$ with generators $ \alpha_1^\pm , \alpha_2^\pm , \ldots , \alpha_k^\pm$. 
Let $X$ be a compact smooth Riemannian manifold with corners and boundary including pairwise disjoint smooth submanifolds $ \partial _1^\pm , \partial _2^\pm , \ldots , \partial _k^\pm$ of $X$ such that each pair $(\partial _j^+,\partial _j^-)$ can be glued smoothly.
Precisely, we consider a compact smooth Riemannian manifold $X_0$ with corners and boundary, as well as smooth submanifolds $ \partial _1, \partial _2, \ldots , \partial _k$ of $X_0$,  such that $X:=(X_0\times \{+,-\})/\sim $, where $(x,\sigma )\sim (x',\sigma ')$ if $(x,\sigma) =(x',\sigma ')$ or $x= x' \in X_0 \setminus \bigcup _{j=1}^k \partial _j$, is a smooth Riemannian manifold with corners and boundary, and set $\partial _j^\sigma = \pi (\partial _j \times \{\sigma\})$ where $\pi$ is the canonical projection from $X$ to $X_0$, i.e.~$\pi (x,\sigma )= x$.
Gluing $\partial _i^+ \times Y$ and $\partial _i^- \times Y$ in the direct product $X \times Y$ by $f_{\alpha _i} \colon Y \to Y$, we   get a fibration $M = M(X,Y,f)$.
More precisely,  define the equivalence relation $\sim$ on $X\times Y$ by $(x, y) \sim (x', y')$ if either $(x, y) = (x', y')$ or there is an $i = 1, 2, \ldots , k$ such that $x \in \partial _i^+$, $x' \in \partial _i^-$ with $\pi (x) =\pi (x')$, and $y' = f_{\alpha _i}(y)$. 
Then the fibration $M$ is defined as the quotient space $(X \times Y)/\sim$. 
Let $\mathcal F$ be the partition of $M$ consisting of  manifolds with corners of the form $(\cup _{a\in G} X
 \times \{ f_a (y)\})/\sim $,
so that $(M, \mathcal F)$ is a foliated compact smooth Riemannian manifold with corners, and the dimension (resp.~codimension) of $\mathcal F$ is that of $X$ (resp.~of $Y$).
 As a usual abuse of notation, throughout this paper, we identify each point $(x,y)\in X\times Y$ with its equivalent class in $M=(X \times Y)/\sim$ whenever it makes no confusion. 
For example, we simply write $\bigcup _{a\in G} X\times \{f_a(y)\}$ for a leaf $(\bigcup _{a\in G} X\times \{f_a(y)\})/\sim $.

The key definition to connect the length averages of $(M, \mathcal F)$ over the fibration $M=M(X,Y,f)$ and the ball averages of $f$ is the following:
 \begin{definition}\label{dfn:0427}
Given a compact Riemannian manifold $X$ with corners, pairwise disjoint subsets $\{ A_i\}_{i=1}^k$ of $X$ with $k\ge 2$ are called \emph{pairwise equidistant} if there is an $R>0$ such that $\inf _{y\in A_j}d(x, y)=R$ for any $x\in A_i$ and $1\le i, j\le k$ with $i\neq j$. 

Furthermore, we say that \emph{$X$ has thin legs with respect to $\{ A_i\}_{i=1}^k$} (or, \emph{$\{ A_i\}_{i=1}^k$ is thin}, for short) if 
there is a non-empty open subset $V$ of $X$ such that 
$m_1+(D-R)<2m_0$ and $2m_1<R+m_0$ where $D=\sup\{ d(x,x') : x\in X, \; x'\in \bigcup_{i=1}^k A_i\}$, $m_0=\inf\{ d(x,x') : x\in V, \; x'\in \bigcup_{i=1}^k A_i\}$ and $m_1=\sup\{ d(x,x') : x\in V, \; x'\in \bigcup_{i=1}^k A_i\}$. 
\end{definition}

Notice that any cylinder is a compact smooth surface with boundary consisting of pairwise equidistant boundary components $A_1, A_2$ (and thus is thin whenever the distance between $A_1, A_2$ is sufficiently larger than the diameters of $A_1, A_2$).
One can show that the pairwise equidistant property for any smooth manifold only holds in this trivial case $k=2$: 
\begin{prop}\label{prop:2.2}
For any integer $k\ge 3$, there is no compact smooth Riemannian manifold $X$  with boundary including pairwise equidistant codimension one smooth submanifolds $A_1, \ldots , A_k$ of $X$.
\end{prop}
 
In contrast, the following theorem ensures that there is a compact smooth surface with corners satisfying the pairwise equidistant property and having thin legs.

\begin{main}\label{prop:pieces}
For any integer $k\ge 2$, there is a compact smooth Riemannian surface $X$ with corners and boundary including pairwise equidistant thin codimension one smooth submanifolds $A_1, \ldots , A_k$ of $X$ each pair of which can be glued smoothly. 
\end{main}
More detailed information on the foliated manifold with corners in Theorem \ref{prop:pieces}, as well as the proof of Proposition \ref{prop:2.2}, will be given in Section \ref{s:thmc}.
We especially refer to Figure \ref{fig:bifurcation}.
Given a finitely generated group action $f=(f_a)_{a\in G}$ by diffeomorphisms on a compact Riemannian manifold $Y$, we introduce
\[
\lambda _G(y) := \limsup _{n\to \infty} \frac{\vert G_n(y) \setminus G_{n-1}(y)  \vert}{\vert G_n(y)  \vert }\qquad (y\in Y),
\]
where $G_n(y):=\{ f_\alpha (y) : \alpha \in G_n\}$ and $A\setminus B$ is the set difference of $A$ and $B$. 
Now we state our final result, which relates the existence of length averages with the underlying group structure.

\begin{main}\label{mainthm:4}
Let $X$ be a compact smooth Riemannian manifold with corners and boundary including pairwise equidistant smooth submanifolds 
$\partial  _1^\pm , \ldots , \partial _k^\pm $ with $k \geq 1$ each of which can be glued smoothly.
Consider a finitely generated group action $f=(f_a)_{a\in G}$ by diffeomorphisms on a compact Riemannian manifold $Y$ with generators $ \alpha_1^\pm , \alpha_2^\pm , \ldots , \alpha_k^\pm$.
 Then, the foliated manifold $(M, \mathcal F)$ with corners over 
the fibration $M=M(X,Y,f)$ 
satisfies the following.
\begin{itemize}
\item[(1)]  (Small boundary case)  If $y\in Y$ is such that $\lambda _G(y)=0$, then for any continuous function $\varphi $ on $M$ and any  $x\in X$, 
  it holds that
\begin{align*}
&\limsup _{r\to\infty} \frac{1}{\vert B_r^{\mathcal F}(x,y) \vert } \int _{B_r^{\mathcal F}(x,y)} \varphi (w) dw  = \frac{1}{\vert X\vert } \limsup _{n\to\infty} \beta _n\widetilde \varphi (y), \\
&\liminf _{r\to\infty} \frac{1}{\vert B_r^{\mathcal F}(x,y) \vert } \int _{B_r^{\mathcal F}(x,y)} \varphi (w) dw  = \frac{1}{\vert X\vert } \liminf _{n\to\infty} \beta _n\widetilde \varphi (y), 
\end{align*}
where $\widetilde \varphi$ is a continuous function on $Y$ given by 
\begin{equation}\label{eq:20240406}
\widetilde \varphi (y) = \int _{X \times \{y\}} \varphi (w )dw.
\end{equation}
\item[(2)]  (Large  boundary case) 
 Assume that $\{\partial _i^\pm\}_{i=1}^k$ is thin. 
Then, there are a non-empty open subset $V$ of $X$ and a continuous function $\varphi $ on $M$ such that the following holds: if $y\in Y$ is such that $\lambda _G(y)>0$, then the length average of $\varphi$ 
 does not exist at $(x,y)$ for every $x\in V$.  
\end{itemize}
\end{main}

Similarly to Theorem \ref{mainthm:2},
the proof of item (2) above ensures that (under the same assumption) if $\lambda_G(y)>0$ then one can find a $\mathcal C^0$ open set $\mathcal U$ of continuous functions such that for any $\varphi\in\mathcal U$ and $x\in V$, the length average of $\varphi$ does not exist at $(x,y)$.
Below, for each $y\in Y$ and $n\in \mathbb N$, we identify $G_n(y)$ with $G_n/\sim_{y}$, where $\sim_{y}$ is an equivalence relation in $G_n$ defined by $a\sim _{y}b$ for $a,b\in G_n$ if and only if $f_a(y)=f_b(y)$.
We further identify each element of $G_n/\sim_{y}$ with its representative. 
 
\begin{rmk}
Let $\lambda_G:=\limsup _{n\to \infty} \frac{\vert G_n \setminus G_{n-1} \vert}{\vert G_n \vert }$.
It is straightforward to see that, given $y\in Y$, we have $G_n(y)=G_n$ for every $n\ge 1$ if and only if $f_\alpha(y)\neq y$ for every non-empty word $\alpha \in \bigcup _{n\ge 1}G_n$. 
These conditions imply $\lambda_G(y)=\lambda _G$ and are satisfied for several examples (e.g.~the group action $f$ given in the proof of Theorem \ref{mainthm:2} and $y$ in a non-empty open subset of $Y$).
\end{rmk}

\begin{rmk}
The condition $\lambda _G =0$ implies that $\{ G_n\}_{n\ge 1}$ is a  F{\o}lner sequence.
Recall that a sequence $\{F_n\}_{n\ge 1}$ of finite subsets of $G$ is 
 said to be \emph{F{\o}lner} if for any $a \in G$,
\[
\lim _{n\to \infty} \frac{\vert a F_n\triangle F_n \vert}{\vert F_n\vert } =0,
\]
where $aF_n =\{ a b : b \in F_n\}$ and $\triangle$ is the symmetric difference operation. 
In fact, for any finite subset $F$ of $G$ and $a,b\in G$, it holds that $\vert ab F\triangle F\vert \le  \vert ab F\triangle a F\vert  +  \vert a F\triangle F\vert \le \vert  b F\triangle F\vert  +  \vert a F\triangle F\vert$ because the (left-)multiplication on $G$ by $a$ is bijective.
Hence, for any $a  \in G$ of the form $a = (a_1a_2\cdots a_m)$ with some $m\ge 1$ and $a_j \in G_1$ for $j=1,\ldots ,m$, by denoting $[a]_j=(a_1a_2\cdots a_j)$, we have 
\begin{align*}
\vert a G_n\triangle G_n\vert &\le \vert a_m G_n\triangle G_n\vert  +  \left\vert [a]_{m-1} G_n\triangle G_n\right\vert \le \cdots \le \sum _{j=1}^m \vert a_j G_n\triangle G_n\vert \\
& \le \sum _{j=1}^m \left(\vert a_j G_n\triangle G_{n+1}\vert +\vert G_{n+1}\triangle G_n\vert \right) \le 2m \vert G_{n+1}\setminus G_n\vert .
\end{align*}
The claim immediately follows from it.

We also note that it was shown in \cite{lindenstrauss2001pointwise} that if $G$ is amenable 
and $\mu$ is an invariant probability measure for $G$ (i.e.~$\mu(f_a^{-1}A) =\mu(A)$ for every $a\in G$ and Borel measurable set $A$), then for any $\mu$-integrable function $\varphi$ and any tempered F{\o}lner sequence $\{F_n\}_{n\ge 1}$,  the average $\lim_{n\to\infty}\frac{1}{\vert F_n\vert } \sum _{a\in F_n} \varphi (f_a(y))$ exists for $\mu$-almost every $y$ (refer to \cite{lindenstrauss2001pointwise} for the precise definition of amenability and tempered F{\o}lner sequences).
This will be an important ingredient to apply the main theorem of \cite{CCSV2021} in the proof of Theorem \ref{mainthm:1}.
\end{rmk}

We close this section by explaining
the idea in the construction of the smooth foliated 3-manifold of Theorem \ref{mainthm:2} (and related examples).
The basic idea
is similar to Theorem \ref{mainthm:4} (2),
that is, to use the non-F{\o}lner property of the holonomy dynamics.  
However, since we cannot construct a smooth model leaf without corners with the equidistant property due to Proposition \ref{prop:2.2},
we need to modify both the model leaf and the gluing of leaves in Theorem \ref{mainthm:4} (2). 
The idea for the modification on the model leaf is to use pants with `long' legs to get sufficiently nearly equidistant boundaries.
The drawback of it is that the nearly equidistant property only holds in the `positive' direction.
This will be overcome by using a special ping-pong dynamics for gluing.

In this connection,
we provide (relatively elementary) examples of smooth foliated 3-manifolds with length averages constructed in the spirit of Theorem \ref{mainthm:4} (1).
One can consider the fibration of a cylinder (a surface with two equidistant boundary components) over a circle diffeomorphism, to which Theorem \ref{mainthm:4} (1) is directly applicable.
Another smooth example whose length averages can be calculated
in a similar manner to
Theorem \ref{mainthm:4} (1) is the following. Let $d\ge 2$, $1\le l\le d-1$ be integers.
Recall that, given an $\mathbb R^{l}$-action $f$ on $\mathbb T^d$ (which is an amenable group action), the orbit foliation by $f$ is a foliation consisting of leaves of the form $\{f(t)(x): t\in\mathbb R^{l}\}$ over $x\in \mathbb T^d$, and the ball average $\beta_r\varphi$ of a continuous function $\varphi$ on $\mathbb T^d$ is given by $\beta_r\varphi(x)=\frac{1}{\vert B_r\vert}\int_{B_r}\varphi(f(t)(x))dt$, where $B_r$ is the $l$-dimensional Euclidean ball of radius $r>0$.
Furthermore, given $\alpha=(\alpha^{(1)},\ldots,\alpha^{(l)})\in (\mathbb R^d)^{l}$, the $\alpha$-rotation $f$ is defined by
$f(t_1,\ldots,t_{l})(x)=x+\sum_{j=1}^{l}t_j\alpha ^{(j)} \mod \mathbb Z^d$.
\begin{prop}\label{prop:2.6}
Let $\alpha=(\alpha^{(1)},\ldots,\alpha^{(l)})\in(\mathbb R^d)^{l}$ be such that 
$\sum_{j=1}^{l}t_j\alpha ^{(j)}
\not\in\mathbb Z^d$ for any $(t_1,\ldots,t_l)\in\mathbb R^l\setminus\{0\}$.
Let $f$ be the $\alpha$-rotation and $\mathcal F$ the orbit foliation by $f$. Then, for any continuous function $\varphi$ on $\mathbb T^d$, $x\in\mathbb T^d$ and $r>0$, it holds that
\[
\frac{1}{\vert B_r^{\mathcal F}(x)\vert } \int _{B_r^{\mathcal F}(x)}\varphi (w) dw = \beta_r\varphi (x)
\]
and
\[
\displaystyle\lim_{s\to\infty}\beta_s\varphi (x) =\int _{\mathbb T^d} \varphi (y)  dy.
\]
In particular, length averages for $\mathcal F$ exist everywhere.
\end{prop}

\subsection*{Organization of the paper}
Sections \ref{s:3}, \ref{s:thmc}, \ref{s:5}, \ref{s:6} will be devoted to the proof of Theorems \ref{mainthm:1}, \ref{prop:pieces},  \ref{mainthm:4}, \ref{mainthm:2}, respectively. 
The proof of Proposition \ref{prop:2.2} will be found in Section \ref{s:thmc}, and  Corollary \ref{cor:y}, Proposition \ref{prop:2.6} in Section \ref{s:6}.

\section{Proof of Theorem \ref{mainthm:1}}\label{s:3}

\subsection{Suspension of an interval exchange transformation with a residual irregular set}\label{S:2.1}
We start the proof of Theorem \ref{mainthm:1} by remembering that Keynes and Newton \cite{KN1976} showed that there is a minimal but non-uniquely ergodic interval exchange transformation $T: S^1 \to S^1$
with $S^1 =\mathbb R/\mathbb Z$.
In fact, $T$ has two distinct ergodic invariant probability measures $\mu_1, \mu_2$, each of which is the push-forward of the Lebesgue measure on the circle by a homeomorphism, so that 
$\mu _1$, $\mu_2$ are non-atomic and 
the supports of $\mu _1$ and $\mu_2$ are dense in $X$ (cf.~\cite[Section 8.1]{Yoccoz2010}). 

Given a bounded positive function $\rho :S^1\to \mathbb R$, let $X_\rho:=\{(x,y):x\in S^1, \; 0\le y\le \rho (x)\}/\approx$, where the equivalence relation $\approx$ is the symmetrization of the binary relation $\sim$ given by $(x,y)\sim (x',y')$ if and only if $(x,y)=(x',y')$ or $x'=T(x)$, $y=\rho(x)$, $y'=0$. 
We identify each point in $X_\rho$ with its representative in the fundamental domain $\{(x,y):x\in S^1, \; 0\le y<\rho (x)\}$.
Recall that the suspension flow $\{f^t\}_{t\in \mathbb R}$ of $T$ over $\rho$ is given as a flow on $X_\rho$ defined by
\[
f^t(x,y)=(T^n(x),y+t-\tau_n(x))
\qquad \text{if $\tau_n(x)\le y+t<\tau_{n+1}(x)$ for $n\in\mathbb Z$},
\]
where $\tau_n(x)$ is the $n$-th hitting time to the roof of $X_\rho$, that is 
\[
\tau_{n+ 1}(x)=\tau_{n}(x)+\rho(T^n(x)),\qquad  \tau_0(x)=0
\]
(recall that $T$ is a bijection).
Let $\mu_{\rho,j}$, $j=1,2$, be the probability measure given by $\mu_{\rho,j}(A)=\frac{(\mu_j \times \mathrm{Leb})(A)}{(\mu_j \times \mathrm{Leb})(X_\rho)}$ for each Borel set $A\subset X_\rho$.
It is straightforward to see that $\mu_{\rho,1}$, $\mu_{\rho,2}$ are also ergodic invariant probability measures for the amenable group action $\{ f^t : t\in \mathbb R\}$. 

Assume that $\rho$ is continuous except the set of all discontinuity points  $x_1,\ldots,x_r$ of $T$, denoted by $D$.
Consider $X_{\rho,0}:= X_\rho \setminus \bigcup _{t\in \mathbb R}f^t (D\times \{0\})$.
Since $D$ is a finite set, $\bigcup _{n\in \mathbb Z}T^n (D)$ is a countable set. 
Hence, each $\mu _{\rho,j}$ cannot be supported on $\bigcup _{t\in \mathbb R}f^t (D\times \{0\})$ because $\mu_j$ is non-atomic.
Therefore, the restriction of $\{ f^t : t\in \mathbb R\}$ on $X_{\rho,0}$ is continuous and it has at least two ergodic invariant measures $\mu_{\rho,1},\mu _{\rho,2}$ whose supports are dense in $X_{\rho,0}$.
Let $\psi _\rho:X_\rho\to \mathbb R$ be a continuous function such that $\int \psi _\rho\, d\mu _{\rho,1}\neq \int \psi _\rho\, d\mu_{\rho,2}$. 
By Birkhoff's ergodic theorem,
there are dense subsets $D_1, D_2$ of $X_\rho$ (containing the supports of $\mu_{\rho,1},\mu_{\rho,2}$) such that
$\lim _{n\to \infty} \psi _n(z)=\int \psi _\rho\, d\mu _{\rho,j}$ for all $z\in D_j$, $j=1,2$, where $ \psi _n(z)=\frac{1}{2n} \int^n_{-n} \psi _\rho\circ f^{t} (z) \, dt$.
Therefore, the sequence $\{\psi_n\}_{n\ge 1}$ of bounded and continuous functions on the Baire space $X_{\rho,0}$ satisfies the hypothesis of 
\cite[Theorem A]{CCSV2021}.
Thus, it follows from \cite[Theorem A]{CCSV2021} that there is a residual subset $R_\rho \subset X_{\rho,0}$ 
 such that 
$\lim _{n\to \infty} \psi_n (z)$ does not exist for all $z\in R_\rho$. Consequently,
    \begin{equation}\label{eq:0819}
\text{$\displaystyle \lim _{T\to \infty} \frac{1}{2T} \int^T_{-T} \psi _\rho\circ f^t (z) \, dt$ does not exist}
 \end{equation}
 for all $z\in R_\rho$.
Notice that $R_\rho$ is also a residual subset of $X_{\rho}$.

\subsection{A basic foliated surface}\label{s:bfs}
We will consider a compact surface induced by $f$ in Section \ref{S:2.1}, which is a key part of Theorem \ref{mainthm:1}. 
The idea is similar to the suspension of $f$, but we need a modification because the suspension of $f$ is not a smooth manifold due to the discontinuity set $D$. 
For the modification, 
we will construct the desired foliated surface by gluing finitely many scalings of a basic foliated compact surface (with boundary). 
We first prepare the basic foliated surface.

Assume that $x_0<x_1<\cdots <x_r$. 
Set $I_j:=[x_j,x_{(j+1 \mod r)}] $ and  
\[
c _0:=  \min _{1\le j\le r} \vert I_j\vert, 
\]
where $\vert I \vert $ is the length of an interval $I$ in $S^1$.
Fix $c \in [\frac{c_0}{2},1)$. 
Let $\mathcal Y_0$ be the linear vector field on $[0,c]\times [-1,2]$ of the form 
\[
\mathcal Y_0(x,y)=(x,-y),
\]
for which a part $\{0\}\times [-1,0)$ of the boundary of $[0,c]\times [-1,2]$ is a leaf, and
let $\mathcal Y_1$ be the constant vector field on $[0,c]\times [-1,2]$ 
\[
 \mathcal Y_1(x,y)=(0,1),
\]
which generates a trivial foliation consisting of vertical leaves. 
Let $B: \mathbb R^2 \to[0,1]$ be a $\mathcal C^\infty$ bump function such that $B(x,y) =0$ if $(x,y) \in [-\frac{c_0}{4}, \frac{c_0}{4}]^2$ and $B(x,y)=1$ if $(x,y)\not\in (-\frac{c_0}{3}, \frac{c_0}{3})^2$ (notice that $[-\frac{c_0}{3}, \frac{c_0}{3}]^2 \subset [-c,c] \times [-1,2]$ because $c  \ge  \frac{c_0}{2}$ by assumption and $c_0=\min _{0\le j\le r} \vert I_j\vert \le 1$). 
Let $\mathcal Y$ be  the $\mathcal C^\infty$ vector field on $[0,c]\times [-1,2]$ given by 
\[
\mathcal Y= (1-B)\mathcal Y_0 +B \mathcal Y_1.
\]
Let $\mathcal F_c'$ be a 
 foliation consisting of integral curves
 generated by $\mathcal Y$ starting from points in  $[0,1]\times \{-1\}$, the singular leaf $\{(0,0)\}$, and the separatrix $\partial_c$ (i.e.~the integral curve including the line segment $(0,\frac{c_0}{4}]\times \{0\}$), 
  see Figure \ref{fig:temp}.
 Let $M_c':=\bigsqcup _{L\in \mathcal F_c'}  L$.
 Notice that  each leaf of $\mathcal F_c'$ intersecting $[0, \frac{c_0}{3}] \times [-1, 2]$ is independent of the choice of $c$ (but depending on $c_0$).

\begin{figure}[h]
\begin{center}
\includegraphics[scale=0.5]{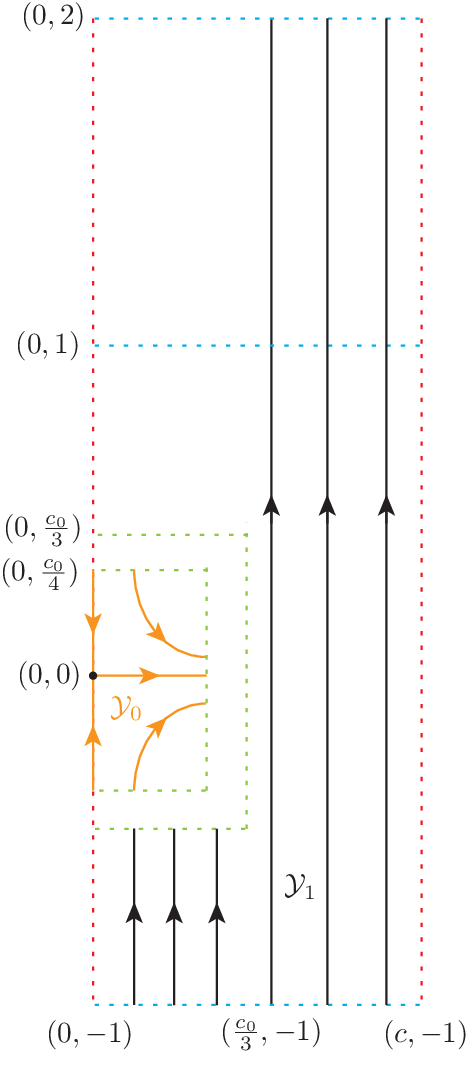}
\includegraphics[scale=0.5]{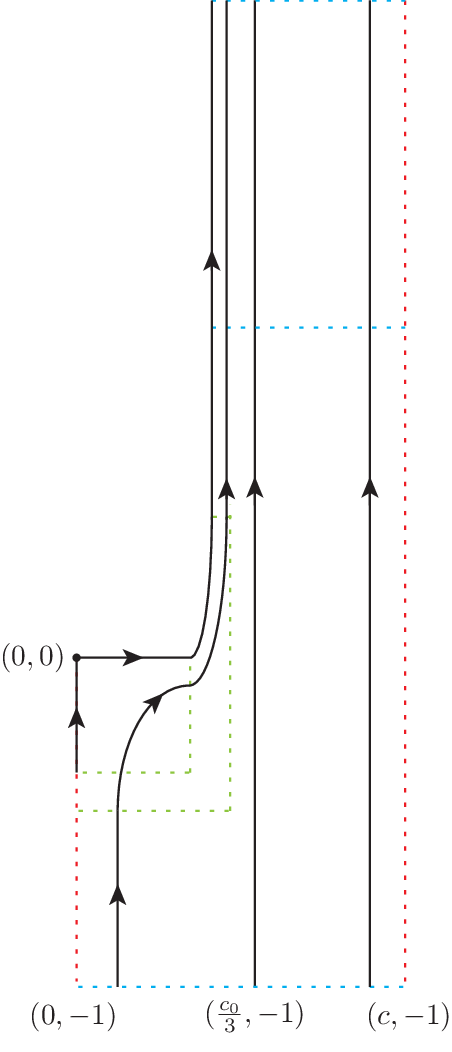}
\mbox{\raisebox{0mm}{\includegraphics[scale=0.265]{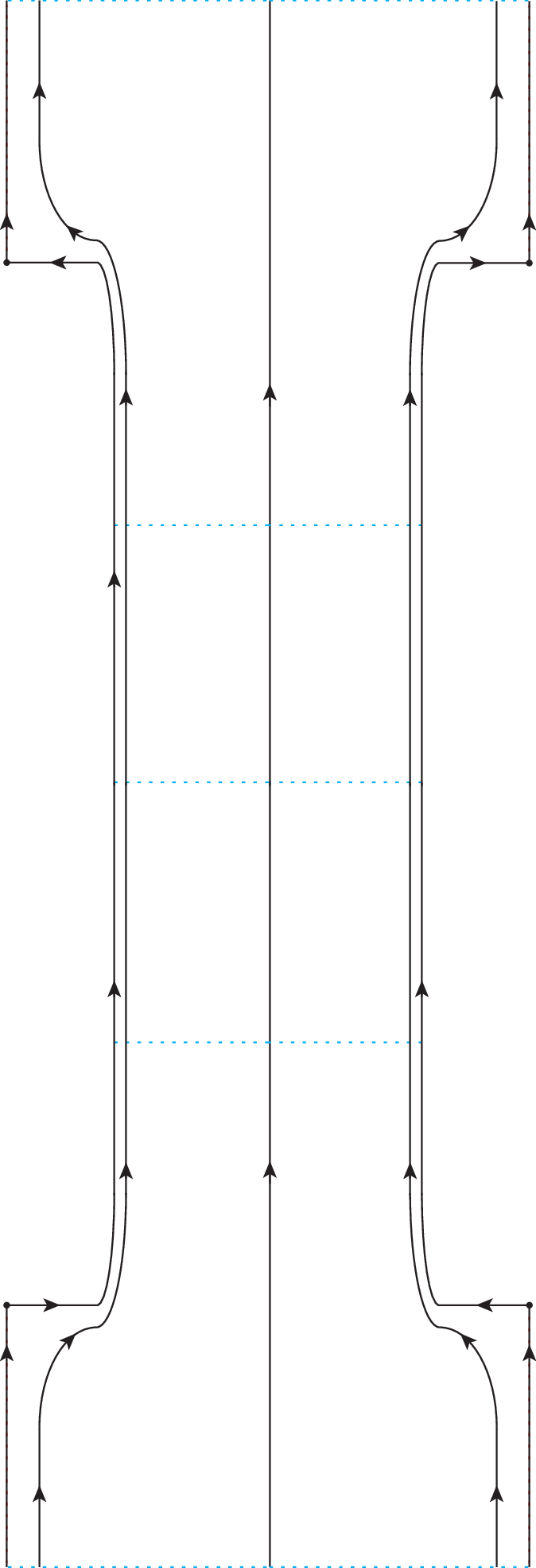}}}
\end{center}
\caption{(a) $\mathcal Y_0$, $\mathcal Y_1$; (b) $(M_c',\mathcal F_c')$; (c) $xy$-projection of $(M_{[a,b]},\mathcal F_{[a,c]})$}
\label{fig:temp}
\end{figure}

\subsection{Removal of auxiliary boundary}\label{sec:3.3}

We will glue the boundary components $[0,c]\times \{-2\}$ and $[0,c]\times \{2\}$ of $M_c'$ by the circle map $T$ in Section \ref{S:2.1},
and use the boundary components $\{0\}\times [-1,0]$ and $\{c\}\times [-1,1]$ 
to glue $M_c'$ with another basic foliated manifold $M_{c'}'$.
However, the resulting foliated manifold will have boundary components corresponding to the separatrices $\partial_c, \partial_{c'}$ of $M_c'$, $M_{c'}'$. 
To remove it, we take the doubling of $M_c'$ along the separatrix before the gluing.
The problem is that the doubling is not smooth, hence
we first need to deform $M_c'$ near the separatrix, as performed below.  

Let $\beta_1 \colon \R \to [0,1]$ be a bump function with $\beta_1^{-1}(1) = \R_{\geq 1}$ and $\beta_1^{-1}(0) =  \R_{\leq \frac{1}{2}}$ such that the restriction $\beta_1 \vert_{[\frac{1}{2},1]}$ is strictly increasing, where $\mathbb R_{\ge r}:=\{x\in\mathbb R: x\ge r\}$.
Define a smooth function $\chi_1 \colon \R \to \R_{\leq 1}$ by $\chi_1(r) := (1- \beta_1(r))(r +\frac{1}{3}) + \beta_1(r)$. 
Then, $\chi_1^{-1}(1) = \R_{\geq 1}$ and $\chi_1\vert_{\R_{\leq \frac{1}{2}}}(r) = r + \frac{1}{3}$ as depicted in the upper left on Figure~\ref{fig:smooth_corners}. 
\begin{figure}[t]
\begin{center}
\includegraphics[scale=0.3]{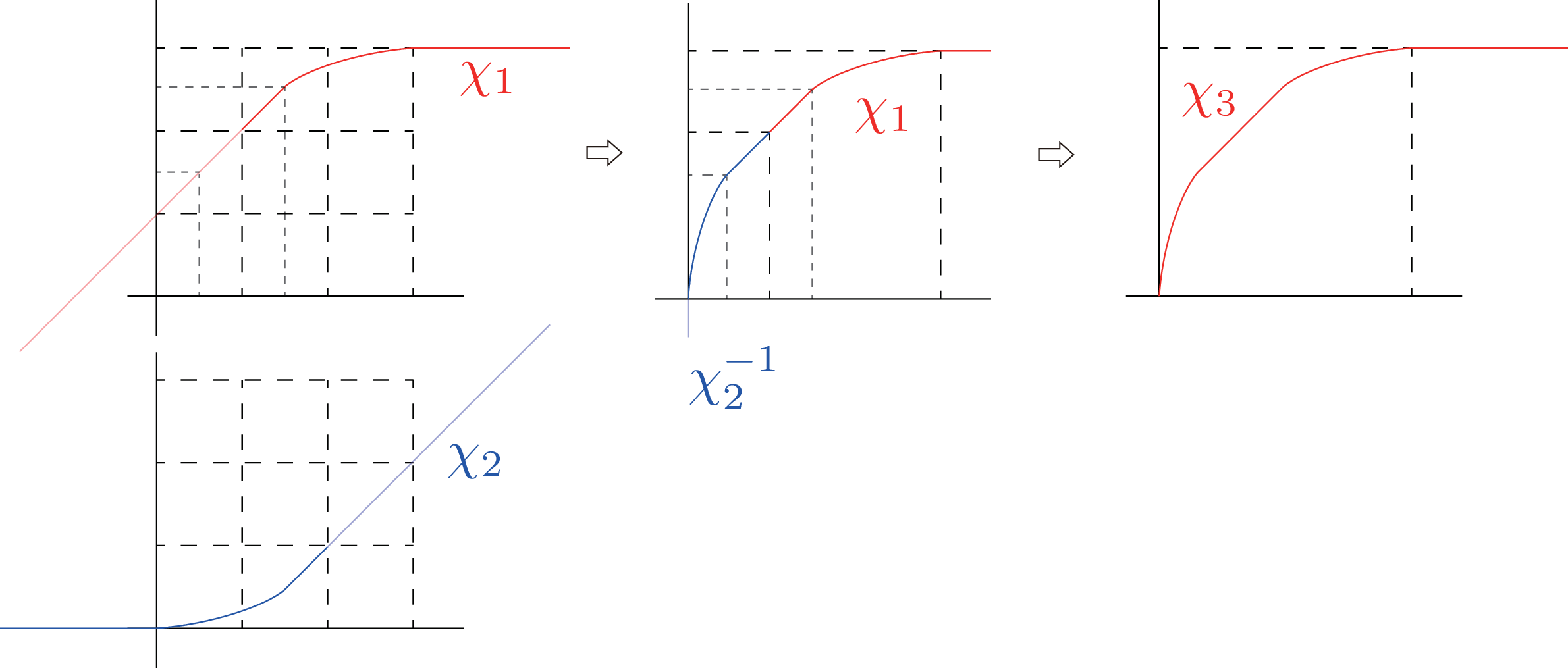}
\end{center}
\caption{Smooth curves for a smooth doubling of $M_c'$}
\label{fig:smooth_corners}
\end{figure}
Similarly, let $\beta_2 \colon \R \to [0,1]$ be a bump function with $\beta_2^{-1}(1) = \R_{\geq \frac{1}{2}}$ and $\beta_2^{-1}(0) =  \R_{\leq 0}$ such that the restriction $\beta_2 \vert_{[0,\frac{1}{2}]}$ is strictly increasing. 
Define a smooth function $\chi_2 \colon \R \to \R_{\geq 0}$ by $\chi_2(r) :=  \beta_2(r)(r -\frac{1}{3})$. 
Then $\chi_2^{-1}(0) = \R_{\leq 0}$ and $\chi_2\vert_{\R_{\geq \frac{1}{2}}}(r) = r - \frac{1}{3}$ as depicted in the lower left on Figure~\ref{fig:smooth_corners}. 
Define a smooth function $\chi_3 \colon [0,2] \to [0,1]$ as 
\[
\chi_3 (r) := 
\begin{cases}
\chi_2 \vert_{\R_{\geq 0}}^{-1}(r) & r \in [0,\frac{1}{2}) \\
\chi_1(r) & r \in [\frac{1}{2}, 2]
\end{cases}.
\]

Recall the separatrix $\partial_c$ of $M_c'$ given in Section \ref{s:bfs}.
Let $U_c$ be a tubular neighborhood of $\partial_c$ in $M_c'$ which is identified with $\partial_c\times [0,2]$ by a smooth diffeomorphism $\alpha :\partial_c\times [0,2] \to U_c$ such that $\alpha(\partial_c\times \{0\})=\partial_c$ and $\alpha(\{ (s,0)\}\times [0,1])$ is a vertical line through $(s,0)$ for any small $s\ge 0$.
Moreover, define a smooth function $\widetilde{\chi} \colon M_c' \to [0,1]$ by 
\[
\widetilde{\chi} \vert_{M_c' \setminus U_c} = 1\quad \text{and}\quad \widetilde{\chi} \vert_{U_c}(\alpha (p,r)) = \chi_3(r) \quad \text{for $(p,r)\in \partial_c\times [0,2]$}.
\]
Then, the graph of $\widetilde{\chi}$, denoted by $\mathrm{graph}(\widetilde{\chi})$, is a smooth compact surface with boundary, and so is $\mathrm{graph}(-\widetilde{\chi})$.
In addition, the union 
\[
 M _c:= \mathrm{graph}(\widetilde{\chi}) \cup \mathrm{graph}(-\widetilde{\chi})
 \]
 is a smooth compact surface with boundary because $\mathrm{graph}(\widetilde{\chi}) $ and $ \mathrm{graph}(-\widetilde{\chi})$ are flat on $\partial_c\times\{0\}=\mathrm{graph}(\widetilde{\chi}) \cap \mathrm{graph}(-\widetilde{\chi})$.
Since the projection $\pi$ given by $\pi(x,y,z)=(x,y)$ for $(x,y,z)\in M_c'\times \mathbb R$ diffeomorphically maps $\mathrm{graph}(\widetilde{\chi}) $ and  $\mathrm{graph}(-\widetilde{\chi})$ to $M_c'$, the surfaces $\mathrm{graph}(\widetilde{\chi})$ and $\mathrm{graph}(-\widetilde{\chi})$ have the induced smooth foliations $\widetilde{\mathcal{F}}_{+}$ and $\widetilde{\mathcal{F}}_{-}$ from $\mathcal{F}_c'$, respectively. 
Since $\partial _c\times \{0\}$ is a leaf of both $\widetilde{\mathcal{F}}_{+} $ and  $\widetilde{\mathcal{F}}_{-} $, the compact surface $M_c$ has the induced smooth foliation $\mathcal{F} _c:= \widetilde{\mathcal{F}}_{+} \cup \widetilde{\mathcal{F}}_{-}$.

\subsection{The foliated compact surface of Theorem \ref{mainthm:1}}

Given a closed interval $[a,b]$ of $S^1$ with length larger than or equal to $c_0$, 
we define a foliated compact surface with boundary $(M_{[a,b]},\mathcal F_{[a,b]})$ in $[a,b]\times [-3,3]\times \mathbb R$ 
\begin{align*}
M_{[a,b]}'&:=T_{(a,-2)}(M_c) \cup T_{(b,-2)}(R_1(M_c)),\\
M_{[a,b]}&:= M_{[a,b]}' \cup R_2(M_{[a,b]}'),\\
 \mathcal F_{[a,b]}&=\{ \text{connected component of } T_{(a,-2)}(L) \cup R_2(T_{(a,-2)}(L)) : L \in \mathcal F_c\} \\
&\cup \{ \text{connected component of } T_{(b,-2)}(R_1(L)) \cup R_2(T_{(b,-2)}(R_1(L))) : L \in \mathcal F_c\},
\end{align*}
with $c= \frac{b-a}{2}$, where $T_{(s,t)}(x,y,z) =(x+s,y+t,z)$ and $R_1(x,y,z)=(-x,y,z)$, $R_2(x,y,z)=(x,-y,z)$, see Figure \ref{fig:temp}.

Note that, the leaves of $\mathcal F_c$ passing through points in $\{(x,-1,\pm 1) : 0\le x \le \frac{c_0}{3}\}$
 do not depend on $c$.
Moreover, the vertical boundary component $\{x=0\} \cap M_c$ of $M_c$ is a union of leaves of $\mathcal F_c$, and any horizontal line $\{y=s, z\ge 0\} \cap M_c$ or $\{y=s, z\le 0\} \cap M_c$ is flat at the boundary $\{x=0\}\cap M_c$ of $M_c$ for each $-1\le s\le 0$ by construction (recall the properties of $\alpha$ in Section \ref{sec:3.3}).
Hence,  for every  $[a,b], [a',b'] \subset S^1$ satisfying that $\vert [a,b] \cap [a',b'] \vert = 1$ and $\min \{ \vert [a,b] \vert , \vert [a',b'] \vert \} \ge  c _0$, one can smoothly glue $(M_{[a,b]},\mathcal F_{[a,b]})$ and $(M_{[a',b']},\mathcal F_{[a',b']})$ along the leaves contained in $\{x \in [a,b] \cap [a',b'], y < 0\}$. 
Similarly, for every  $[a,b], [a',b'] \subset S^1$ satisfying that $\vert \overline{T((a,b))} \cap \overline{T((a',b'))} \vert = 1$ and $\min \{ \vert [a,b] \vert , \vert [a',b'] \vert \} \ge  c _0$, one can smoothly glue $(M_{[a,b]},\mathcal F_{[a,b]})$ and $(M_{[a',b']},\mathcal F_{[a',b']})$ along the leaves contained in $\{x \in \overline{T((a,b))} \cap \overline{T((a',b'))}, y > 0\}$, where $T$ is the circle map of Section \ref{S:2.1}. 

Let $(M_0,\mathcal F_0)$ be the foliated compact surface with boundary obtained by smoothly gluing  $(M_{I_j},\mathcal F_{I_j}) $ with $(M_{I_{(j+1 \mod r)}},\mathcal F_{I_{(j+1 \mod r)}})$ along the boundary components in $\{ y<0 \}$, and gluing $(M_{T(I_j)},\mathcal F_{T(I_j)}) $ with $(M_{T(I_k)}, F_{T(I_k)})$ along the boundary components in $\{ y>0 \}$ if $\vert \overline{T(I_j)} \cap \overline{T(I_k)} \vert =1$ for $j=1,\ldots ,r$. 
Let $M$ be the compact surface obtained by gluing the boundary of $M_0$ by the identity map on $S^1$, 
that is, $M$ is the quotient space of $M_0$ by the symmetrization of the relation $\sim$ given by $(x,y,z)\sim (x',y',z')$ if and only if $(x,y,z)= (x',y',z')$ or $x'=x$, $y=3$, $y'=-3$, $z=z'$.
We identify $(x,y,z) \in M_0$ with its equivalent class in $M$.


\subsection{Non-existence of length averages}
For each $x\in S^1$, let $\ell (x)$ be the length of the leaf of $\mathcal F_0$ through the point $(x,-3,1)$.
Then, $\ell$ is a bounded and positive function on $S^1$, and continuous except $D$, so we can apply the result in Section \ref{S:2.1} to the case $\rho=\ell$.
Let $\alpha_+:X_\ell \to \alpha_+(X_\ell)$
be the diffeomorphism defined as $\alpha_+(x,y)$ is the point in the leaf $\mathcal F_0((x,-3,1))$
and the distance between $\alpha_+(x,y)$ and $(x,-3,1)$ in $\mathcal F_0((x,-3,1))$ is $y$. 
Similarly, define $\alpha_-:X_\ell \to \alpha_-(X_\ell)$ with $(x,-3,-1)$ instead of $(x,-3,1)$.
Then, it is straightforward to see that
for any $t\in \mathbb R$ such that $0\le y+t<\ell (x)$,
the distance between $\alpha_\pm(x,y)$ and $\alpha_\pm(f^t(x,y))$ in $\mathcal F_0((x,-3,z))$ is $t$.
Therefore, it follows from the construction of $\{f^t\}_{t\in\mathbb R}$ that
\begin{equation}\label{eq:20240809}
\frac{1}{\vert B_r^{\mathcal F}(\alpha_\sigma(x,y)) \vert } \int _{B_r^{\mathcal F}(\alpha_\sigma(x,y))} \varphi (w) dw=\frac{1}{2r} \int^r_{-r} (\varphi\circ \alpha_\sigma) \circ f^t (x,y) \, dt
\end{equation}
for any $\sigma\in\{+,-\}$, $r>0$ and continuous function $\varphi :M\to\mathbb R$ (notice that $(x,y)$ is not a periodic point of $\{f^t\}_{t\in\mathbb R}$ because $T$ is minimal).

Let $\psi _\ell: X_{\ell}\to \mathbb R$ be the continuous function and let $R_\ell$ be the residual subset of $X_{\ell}$ given in Section \ref{S:2.1}.
Let $\varphi :M\to \mathbb R$ be the continuous function that is uniquely determined by
\[
\varphi (\alpha_\sigma(x,y))= \psi _\ell(x,y) \quad \text{for $(x,y)\in X_\ell$, $\sigma\in\{+,-\}$}.  
\]
Then, for any $(x,y,z)$  in the residual set
\[
R:=\left\{ \alpha_+(x,y)  : (x,y)\in R_\ell\right\}\cup\left\{ \alpha_-(x,y)  : (x,y)\in R_\ell\right\}
\]
of $M$,
the length average of $\varphi$ along the leaf $\mathcal F(x,y,z)$  does not exist: 
it follows from \eqref{eq:20240809} that if $(x,y,z)\in R$, then
\[
\frac{1}{\vert B_r^{\mathcal F}(x,y,z) \vert } \int _{B_r^{\mathcal F}(x,y,z)} \varphi (w) dw=
\frac{1}{2r} \int^r_{-r} \psi_\ell \circ f^t (x,y) \, dt,
\]
which does not converge in the limit $r\to\infty$ due to \eqref{eq:0819}.

\section{Proof of Proposition \ref{prop:2.2} and Theorem \ref{prop:pieces}}\label{s:thmc}

\subsection{Proof of Proposition \ref{prop:2.2}}

Let $X$ be a compact Riemannian manifold with pairwise equidistant boundary components $A_1,\ldots , A_k$ with $k \ge 2$. 
We will show that $k=2$. 
Denote by $R$ the distance between each pair $(A_i, A_j)$. 
We first show the following.
\begin{lemma}\label{lem:2.3}
Let $x\in A_i$, $y\in A_j$ satisfying $d(x,y)=R$ and let $\gamma :[0,R]\to X$ be
a parametrization of a shortest path from $x$ to $y$ with unit speed. 
Then, $\gamma$ is a smooth geodesic which
does not intersect with 
$\bigcup _{l=1}^k A_l$ except $x, y$, and orthogonal to $A_i$ at $x$ and to $A_j$ at $y$.
\end{lemma}
\begin{proof}
Note that the shortest path does not intersect with $\bigcup _{l=1}^k A_l$ except $x=\gamma (0)$ and $y=\gamma (R)$.  (Otherwise, the distance between some pair $(A_{i'}, A_{j'})$ should be strictly less than $R$.) 
Thus, $\gamma$ is a smooth geodesic.

Define $h:X\to \mathbb R$ by $h(z)=d(x,z)$. 
Then, $h$ is a smooth function around $A_j$ and  the restriction $h\vert _{A_j}$ takes the minimal value $R$ at $y$. Thus, $y$ is a critical point of $h\vert _{A_j}$. 
In particular, 
 $T_y A_j$ belongs to 
  $\mathrm{Ker} \, D_yh$.
On the other hand, since $\gamma $ is a geodesic, Gauss Lemma (\cite{MR182537,klingenberg1995riemannian}) implies that $\mathrm{Ker}\, D_yh$ is orthogonal to $\frac{d\gamma }{dt}$ at $y$.
Hence, $A_j$ and $\gamma$ are orthogonal at $y$. Similarly, $A_i$ and $\gamma$ are orthogonal at $x$.
\end{proof}

Let us prove Proposition \ref{prop:2.2}. 
By the compactness, for each $x\in A_1$ and $j\ge 2$, there is $y_j\in A_j$ such that $d(x,y_j)=R$.
Denote by $\gamma _j$ the geodesic given in Lemma \ref{lem:2.3} connecting $x$ and $y_j$.
Then, $A_1$ and $\frac{d\gamma _j}{dt}$ are orthogonal at $x$ for each $j$ by Lemma \ref{lem:2.3}.
Hence, it follows from the uniqueness of the geodesic starting at a given point with a given speed that $\gamma _j = \gamma _{j'}$ for any $j, j'$.
That is, $k=2$.

\subsection{A pair of pants with equidistant boundary components}\label{ss:4.2}
The rest of Section \ref{s:thmc} is devoted to the proof of  Theorem \ref{prop:pieces}. 
We first consider the compact smooth surface $X_0$ with corners, contained in $\mathbb R^2$, given in Figure \ref{fig:bifurcation}.
Here, the solid line means the boundary (with corners) of $X_0$, and the distance between the two closest corners is $1$.

   \begin{figure}[h]
\begin{center}
\includegraphics[scale=0.4]{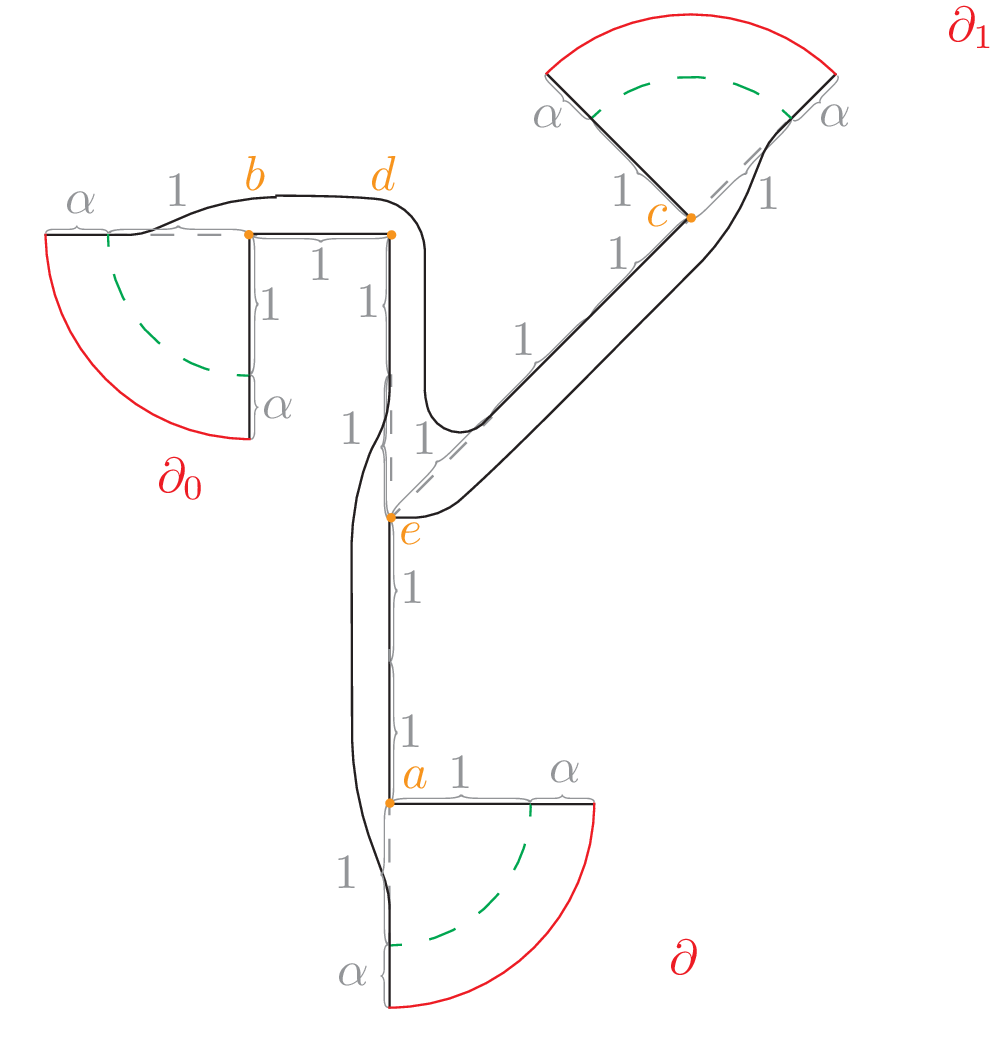}
\end{center}
\caption{Compact surface $X_0$ with corners with equidistant property}
\label{fig:bifurcation}
\end{figure}

 More precisely, let $T_a$ be the translation by a vector $a\in \mathbb R^2$ (i.e.~$T_a(z) =z+a$) and $R_\theta$ the rotation of angle $\theta$, $\tilde F=\{ (x,y)\in \mathbb R^2 : x^2 +y^2 \le (1+\alpha )^2, \; x, y \ge 0\}$ with some $\alpha >0$, and
 \[\label{pass20241017}
 \tilde P=\left\{(x,y)\in \mathbb R^2 : -1\le y \le 3, \; \chi  (y) \le x \le 0\right\},
 \]
where $\chi : \mathbb R\to [-\frac{1}{3},0]$ is a $\mathcal C^\infty$ function such that $\chi  (y)<0$ if and only if $y\in (-1, 3)$, and $\chi  ^{(k)}(-1)=\chi  ^{(k)}(3)=0$ for all $k\ge 1$.  
 Let
 \begin{itemize}
 \item $e=(0,0)$, $ a=(0,-2)$, $d=(0,2)$, $b=(-1,2)$,  $ c=(\frac{3}{\sqrt 2},\frac{3}{\sqrt 2})$;
 \item $ F=T_{ a} \circ R_{\frac{3\pi}{2}} (\tilde F)$, $ F_0=T_{b}\circ R_\pi (\tilde F)$, $ F_1=T_{ c} \circ R_{\frac{\pi}{4}} (\tilde F)$;
 \item $ P=T_{ a} (\tilde P)$, $ P_1=T_{ c}\circ R_{\frac{3}{4}\pi}(\tilde P) $.
 \end{itemize}
 Furthermore, let $\chi_1 :\mathbb R\to \mathbb R$ be a $\mathcal C^\infty$ function such that $\chi_1 (x) >2$ on $(-2,0]$, $\chi_1 (x) > x$ on $(0,\frac{3}{\sqrt 2})$, $\chi_1  ^{(k)}(-2)=\chi_1  ^{(k)}(\frac{3}{\sqrt 2})=0$ for all $k\ge 1$, and the distance  between $d$ and $c$ in 
 \[
 Q:=\left\{ -2 \le x \le 0, \; 2\le y \le \chi_1 (x)\right\} \cup \Big\{  0 \le x \le \frac{3}{\sqrt 2}, \; x\le y \le \chi_1 (x)\Big\} 
 \]
 is $4$.
Set 
  $X_0 =  F \cup  F_0 \cup  F_1 \cup  P \cup   P_1 \cup Q$. 
 Then, by construction, it is straightforward to see that
the distances between 
\begin{itemize}
\item any point in $ \partial :=T_{ a} \circ R_{\frac{3\pi}{2}} (\tilde \partial )$ and any point in $ \partial _0:=T_{ b} \circ R_\pi (\tilde \partial )$, where $\tilde \partial = \{ (x,y)\in \mathbb R^2 : x^2 +y^2 = (1+\alpha )^2, \; x, y \ge 0\}$;
\item any point in $ \partial $ and any point in $ \partial _1:=T_{ c} \circ R_{\frac{\pi}{4}} (\tilde \partial )$;
\item any point in $ \partial _0$ and any point in $ \partial _1$
\end{itemize}
are $7+2\alpha$. In particular, $X_0$ has pairwise equidistant boundary components. 
 
The obstruction to make a compact surface with pairwise equidistant boundary components from copies of $X_0$ by gluing the copies of the boundary components $ \partial ,  \partial _0,   \partial _1$ of $X_0$ is that 
 the boundary components cannot be glued smoothly as a subset of $\mathbb R^2$ (with the Euclidean distance).
 Thus, we modify the construction of $X_0$ as a compact surface in $\mathbb R^3$ as follows.
Instead of $\tilde F$, consider
\[
\tilde F' =\left\{ (\rho(r,\theta ) , \chi _2(r)) \in \mathbb R^3 :  0\le \theta \le \frac{\pi}{2}, \; 0\le r\le (1+\alpha )\right\}
\]
where $\rho (r,\theta )=(r\cos \theta,r\sin \theta )$ and $\chi _2: [0,2) \to [0,-1]$ is a non-increasing $\mathcal C^\infty$ function such that $\chi _2(r) =0$ if and only if $r\in [0,1]$, $\lim _{r\to 2-0}\chi _2(r) =-1$ and the inverse map $\psi: [0,-1 )\to [1,2)$ of the restriction of $\chi _2$ on $[1,2)$ satisfies that $\lim _{z\to -1-0}\psi ^{(k)}(z)=0$ for all $k\ge 1$, as same as the smooth function $\chi_3$ in \S\ref{sec:3.3}. 
See  Figure \ref{fig:bifurcation2}.

 \begin{figure}[h]
\begin{center}
\includegraphics[scale=0.4]{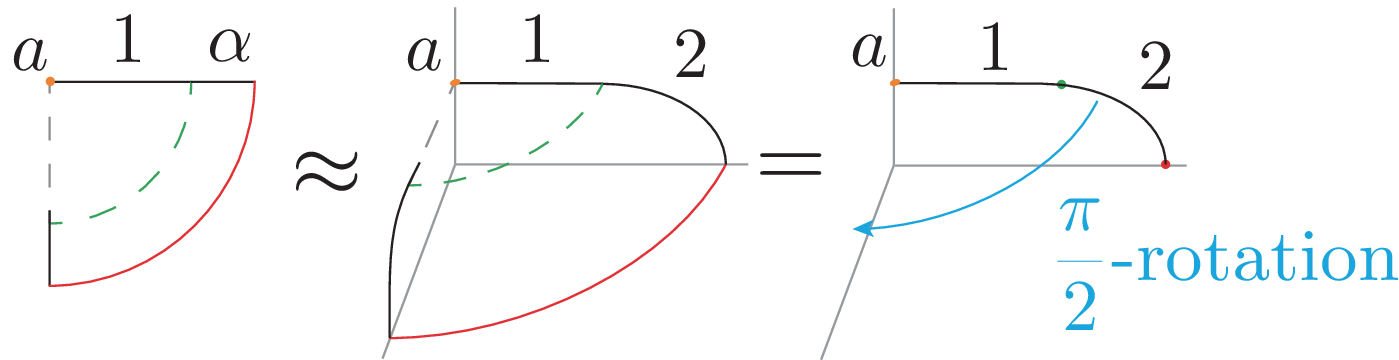}
\end{center}
\caption{Deformation of $X_0$ to obtain $X_0'$}
\label{fig:bifurcation2}
\end{figure}

Let $T_a'(x,y,z)= (T_a(x,y),z)$ and $R_\theta '(x,y,z) =(R_\theta (x,y),z)$, and define $X_0'$ as in the construction of $X_0$ with $F', T_a', R_\theta '$ instead of $F, T_a, R_\theta$. 
Similarly, define boundary components $ \partial ',  \partial _0',   \partial _1'$ of $X_0'$ as in the construction of $ \partial ,  \partial _0,   \partial _1$. 
By construction, the boundary components $ \partial ',  \partial _0',   \partial _1'$ are pairwise equidistant and can be glued smoothly.

\subsection{A compact surface with equidistant boundary components}

For  each $k\ge 2$, we will make a compact smooth Riemannian surface $X$ with corners and with boundary including 
pairwise equidistant codimension one smooth submanifolds $A_1, \ldots , A_k$ of $X$ each pair of which can be glued smoothly.
In the case when $k=2$, one can just take $X$ as the cylinder $S^1 \times [0,1]$, so we may assume $k\ge 3$.

Let $n$ be a positive integer  satisfying that $k\le 2^n+1$, and set  $N=1+2+\dots +2^{n}$.
Consider $N$-copies $X^{m,\ell}$ of $X_{0}'$ with $0\le m\le n$ and $0\le \ell \le 2^{m}-1$. 
Let $ \partial ^{m,\ell },  \partial _{0}^{m,\ell },   \partial _{1}^{m,\ell }$ be the boundary components of $X^{m,\ell}$ corresponding to $ \partial ',  \partial _0',   \partial _1'$.
We glue $\partial _s^{m,\ell }$ with $\partial ^{m+1, 2\ell +s}$ with $s\in \{0,1\}$
by the canonical equivalent relation.
For example, glue $\partial _0^{m,\ell }$ with $\partial ^{m+1, 2\ell}$ by  
\[
T_{a}^{m,\ell}\circ R_{\frac{3\pi}{2}} ^{m,\ell } (x,y,z) \sim T_{ b} ^{m+1,2\ell }\circ R_{\pi} ^{m+1,2\ell } (x,y,z)  \quad \text{for $(x,y,z) \in \tilde \partial '$},
\]
where $T_{\alpha}^{m,\ell }$, $R_{\theta }^{m,\ell}$ are the $(m,\ell)$-th copy of $T_\alpha'$,  $R_\theta '$ and $\tilde \partial '$ is a boundary component of $\tilde F'$ corresponding to $\tilde \partial$.
Denote the resulting compact surface by $X_1$.
By construction, the boundary of $X_1$ includes $2^n+1$ codimension one smooth submanifolds $\partial ^0, \partial ^1, \ldots , \partial ^{2^n}$ such that the distance between any point in $\partial ^0$ and any point in $\partial ^j$ with $j\in \{1,\ldots ,2^n\}$ is $nr_0$, where $r_0$ is the distance betweeb $\partial '$ and $\partial '_0$.
See Figure \ref{fig:plug}.

 \begin{figure}[h]
\begin{center}
\includegraphics[scale=0.26]{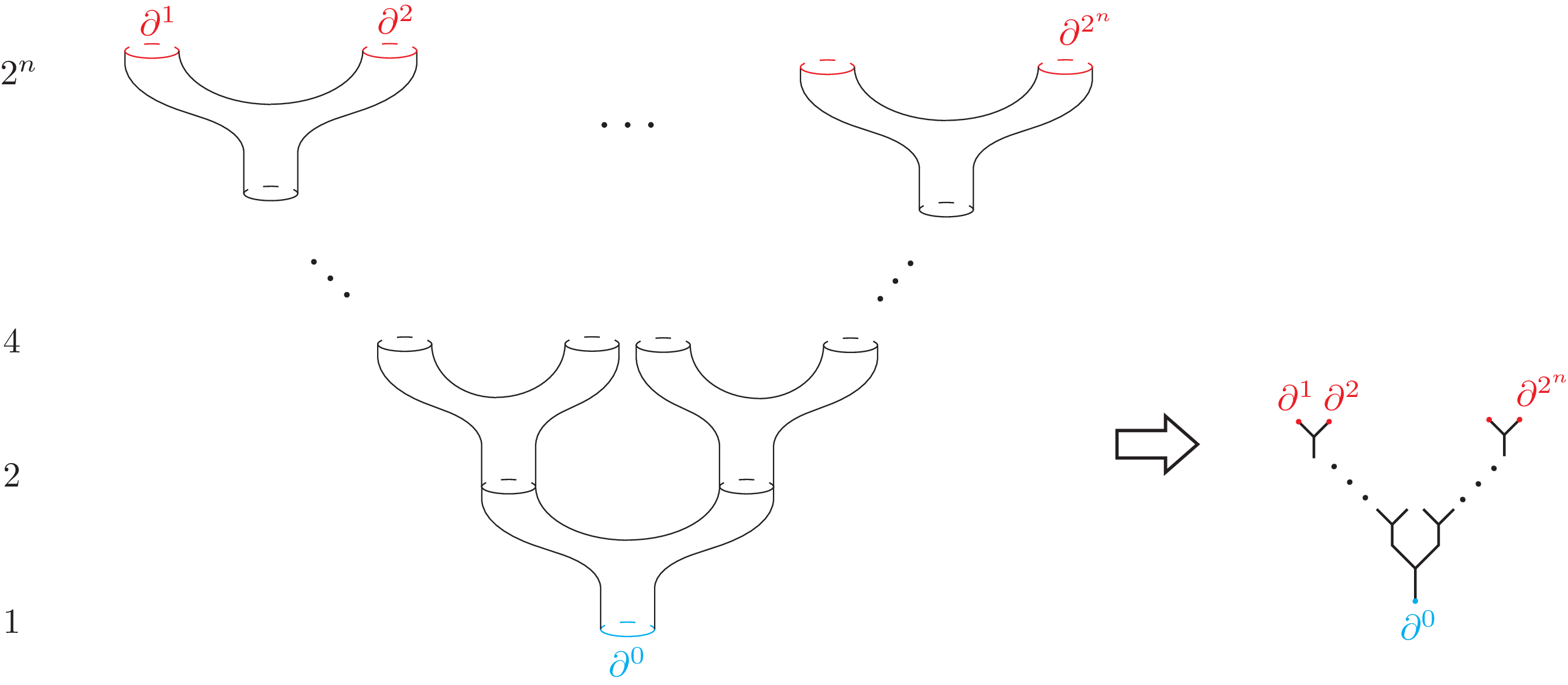}
\end{center}
\caption{Gluing $N$-copies of $X_0'$ to obtain $X_1$}
\label{fig:plug}
\end{figure}

Finally, take $k$-copies $X_{(0)}, \ldots , X_{(k-1)}$ of $X_1$ and denote the boundary components of $X_{(j)}$ corresponding to $\partial ^0, \ldots , \partial ^{k-1}$ (recall $k-1\le 2^n$)   by
$\partial ^0_{(i)},  \ldots , \partial ^{k-1}_{(i)}$.
Gluing $\partial ^{i}_{(j)}$ with $\partial ^{j+1}_{(i)}$ for every $i, j \in \{0,\ldots ,k-1\}$ with $i>j$, we obtain a compact surface $X$ with the desired equidistant property, as depicted in Figure \ref{fig:plug_comb}.
In fact, for every $i, j \in \{0,\ldots ,k-1\}$, the distance between any point in $\partial ^0_{(j)}$ and any point in $\partial ^0_{(i)}$ is $2nr_0$, and one can smoothly glue $\partial ^0_{(j)}$ and $\partial ^0_{(i)}$. 

 \begin{figure}[h]
\begin{center}
\includegraphics[scale=0.4]{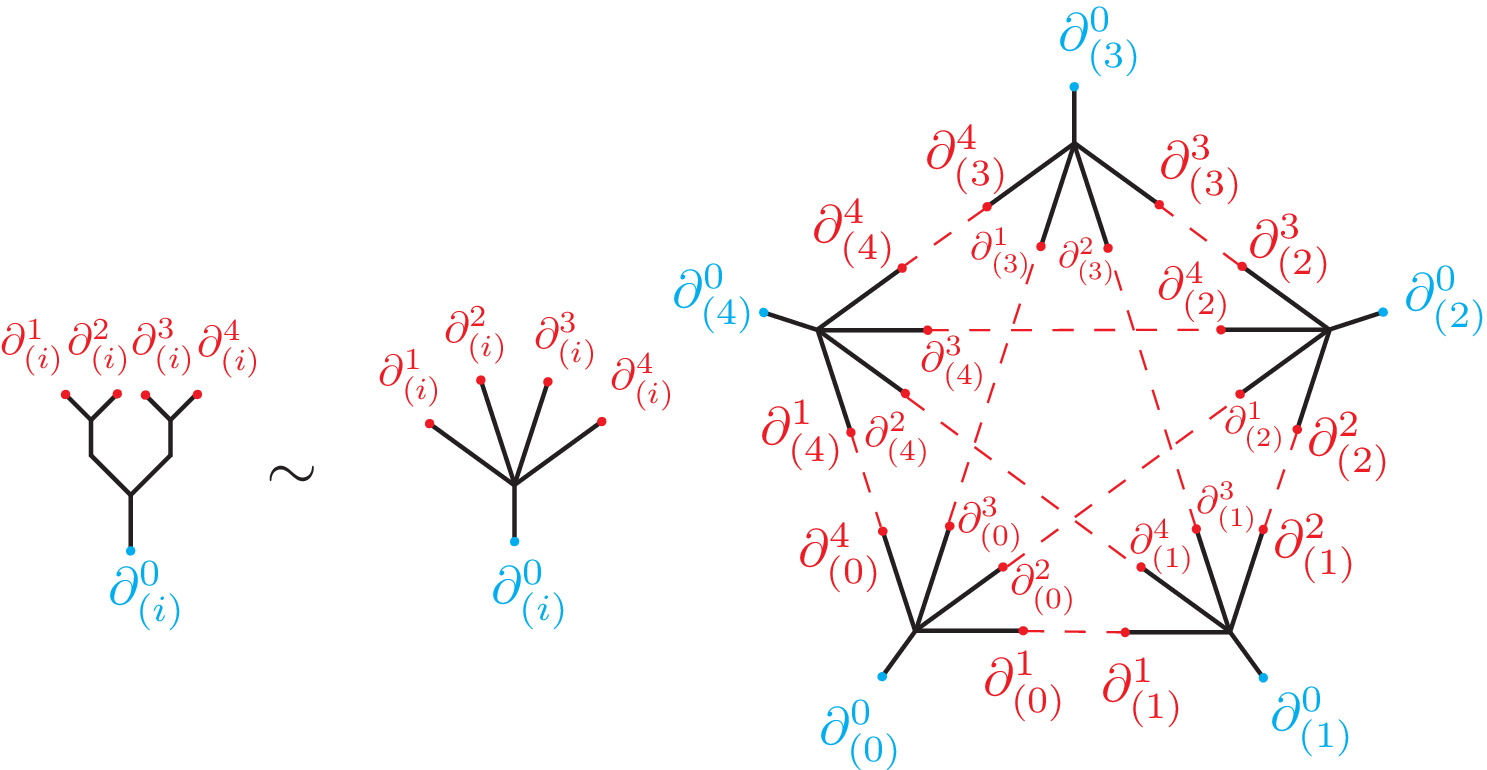}
\end{center}
\caption{Gluing $k$-copies of $X_1$ to obtain $X$ (in the case $k=5$)}
\label{fig:plug_comb}
\end{figure}

\subsection{Thin legs}
The compact surface $X$ in the previous section does not necessarily have thin legs with respect to $\{ A_i\}_{i=1}^k$, but one can easily modify it to get the thin property.
In fact, consider the modification $\widetilde X_0'$ of $X_0'$ by making the distance between the points $a$ and $e$ of $X_0$ sufficiently larger 
(recall page \pageref{pass20241017} and notice that the distance of $\partial$, $\partial _0$ and the distance of $\partial$, $\partial _0$ are still same), and in the construction of $X_1$, glue a copy of $\widetilde X_0'$ (instead of that of $X_0'$) in the step of gluing $X^{0,0}$.
Then, by taking the interior of the union of all elements except the $X_{0,0}$ part in every $X_{(j)}$'s (copy of $X_1$) as $V$ of Definition \ref{dfn:0427} (for the example $k=5$ in Figure \ref{fig:plug_comb}, $V$ is the area inside the pentagon),  the desired estimates $m_1+(D-R)<2m_0$ and $2m_1<R+m_0$ in Definition \ref{dfn:0427} immediately follow. 

\section{Proof of Theorem \ref{mainthm:4}}\label{s:5}

Let $X$ be a compact smooth Riemannian manifold $X$ with corners and boundary consisting of
pairwise equidistant codimension one submanifolds 
$\partial  _1^\pm , \ldots , \partial _k^\pm $ with $k\ge 1$ with distance $R$.
Let $f=(f_a)_{a\in G}$  be a finitely generated group action by diffeomorphisms on a compact Riemannian manifold $Y$ with generators $ \alpha_1^\pm , \alpha_2^\pm , \ldots , \alpha_k^\pm$, 
and $(M, \mathcal F)$ the foliated smooth manifold 
with corners 
over the fibration $M=M(X,Y,f)$. 
Recall that we identified an equivalent class of $M$ with its representative (e.g.~$X\times \{y\}$ for $(X\times \{y\})/\sim$).

\subsection{Small boundary case}

Given $r>0$,  let $n$ be the positive integer satisfying $nR\le r <(n+1)R$.
Then, it holds that
\[
 \bigsqcup _{a \in G_{n-1}(y)} X \times \{ f_a (y)\} \subset B_r^{\mathcal F}(x,y) \subset \bigsqcup _{a \in G_{n+1}(y)} X \times \{ f_a (y)\}.
\]
Therefore,
\[
\vert X\vert \, \vert G_{n-1}(y)\vert  \le \vert  B_r^{\mathcal F}(x,y) \vert \le \vert X\vert \, \vert G_{n+1}(y)\vert 
\]
and
\[
 \sum _{a \in G_{n-1}(y)} \widetilde \varphi \circ f_a (y)  \le \int _{  B_r^{\mathcal F}(x,y) }\varphi (w) dw \le  \sum _{a \in G_{n+1}(y)} \widetilde \varphi \circ f_a (y),
\]
where $\widetilde \varphi$ is defined as in \eqref{eq:20240406}.
Hence, 
\begin{multline*}
\frac{1}{\vert B_r^{\mathcal F}(x,y) \vert } \int _{B_r^{\mathcal F}(x,y)} \varphi (w) dw \le \frac{1}{\vert X\vert \, \vert G_{n-1}(y)\vert } \sum _{a \in G_{n+1}(y)} \widetilde \varphi \circ f_a (y)\\
 = \frac{1}{\vert X\vert }  \left( 1 - \frac{\vert G_{n+1}(y)\vert - \vert G_{n-1}(y) \vert }{\vert G_{n+1}(y)\vert } \right)^{-1} \frac{1}{\vert G_{n+1}(y)\vert } \sum _{a \in G_{n+1}(y)} \widetilde \varphi \circ f_a (y),
\end{multline*}
and similarly
\begin{multline*}
\frac{1}{\vert B_r^{\mathcal F}(x,y) \vert } \int _{B_r^{\mathcal F}(x,y)} \varphi (w) dw \\
\ge\frac{1}{\vert X\vert }  \left( 1 - \frac{\vert G_{n+1}(y)\vert - \vert G_{n-1}(y) \vert }{\vert G_{n+1}(y)\vert } \right) \frac{1}{\vert G_{n-1}(y)\vert } \sum _{a \in G_{n-1}(y)} \widetilde \varphi \circ f_a (y).
\end{multline*}
Notice that, since $\vert G_{n+1}(y)\vert \ge \vert G_n(y)\vert$, 
\[
\frac{\vert G_{n+1}(y)\vert -\vert G_{n-1}(y)\vert}{\vert G_{n+1}(y)\vert } \le \frac{\vert G_{n+1}(y)\vert -\vert G_{n}(y)\vert}{\vert G_{n+1}(y)\vert } +\frac{\vert G_{n}(y)\vert -\vert G_{n-1}(y)\vert}{\vert G_{n}(y)\vert }, 
\]
so that   $\lambda_G(y)=0$  implies that  $\lim _{n\to \infty} \frac{\vert G_{n+1}(y)\vert -\vert G_{n-1}(y)\vert}{\vert G_{n+1}(y)\vert } =0$.
From these observations, we immediately get the conclusion of Theorem \ref{mainthm:4}~(1).

\subsection{ Large  boundary case}\label{S:41}

Assume that $\{\partial ^\pm_i\}_{i=1}^k$ is thin, that is, 
 one can find  a non-empty open set $V\subset X$ such that $m_1+(D-R)<2m_0$ and $2m_1<R+m_0$ with $D=\sup\{ d(x,x') : x\in X, \; x'\in \partial\}$, $m_0=\inf\{ d(x,x') : x\in V, \; x'\in \partial \}$ and $m_1=\sup\{ d(x,x') : x\in V, \; x'\in \partial\}$,
 where $\partial :=\bigcup_{i=1}^k (\partial ^+_i\cup \partial ^-_i)$.
Take $R_0\in (m_1+(D-R), 2m_0)$ and $R_1 \in (2m_1, R+m_0)$. 
Set
\begin{align*}
r_n:=R_0+(n-1)R , \quad 
 s_n:=R_1 +(n-2)R,
\end{align*}
so that
\begin{equation}\label{eq:0407d}
\begin{split}
& m_1+(n-2)R+D< r_n < m_0+(n-1)R+m_0, \\
&  m_1+(n-2)R+m_1 < s_n < (n-1)R+m_0.
\end{split}
\end{equation}
Then, for any $x\in V$, $y\in Y$, we have
\begin{itemize}
\item[(i-a)] $X\times \{ f_a(y)\}\subset B_{r_n}^{\mathcal F}(x,y)$ for any $a\in G_n(y)$;
\item[(i-b)] $(V\times \{ f_a(y)\})\cap B_{r_n}^{\mathcal F}(x,y)=\emptyset$ for any $a\in G_{n+1}(y)\setminus G_n(y)$;
\item[(ii-a)] $V\times \{ f_a(y)\}\subset B_{s_n}^{\mathcal F}(x,y)$ for any $a\in G_n(y)$;
\item[(ii-b)] $(X\times \{ f_a(y)\})\cap B_{s_n}^{\mathcal F}(x,y)=\emptyset$ for any $a\in G_{n+1}(y)\setminus G_n(y)$.
\end{itemize}
In fact, $B_{m_1}(x,y)$ contains $X\times \{y\}$, implying $B_{m_1+(n-2)R}(x,y)$ includes $\partial \times \{f_a(y)\}$ for each $a\in G_{n-1}(y)$. Thus one gets $X\times \{f_a(y)\} \subset B_{m_1+(n-2)R+D}(x,y)\subset B_{r_n}(x,y)$  for each $a\in G_{n}(y)$, 
which implies (i-a).
For (i-b), notice that for each $a\in G_{n+1}(y)\setminus G_n(y)$ and $x'\in V$, by denoting minimizers of $\min_{x''\in\partial}d_X(x',x'')$ and  $\min_{x''\in\partial}d_X(x,x'')$ by $x_0$ and $x_1$, respectively, one has
\begin{align*}
&d_{(x,y)}
((x',f_a(y)),(x,y))\\
& = 
 d_{(x,y)}((x',f_a(y)), (x_0,f_a(y))) +  d_{(x,y)}((x_0,f_a(y)),(x_1,y)) +  d_{(x,y)}((x_1,y),(x,y)) \\
 &\ge m_0+(n-1)R+m_0>r_n,
\end{align*}
where $d_{(x,y)}$ is the distance on $\mathcal F(x,y)$.
(ii-a) and (ii-b) can be shown in similar manners (recall \eqref{eq:0407d}).

Given a point $(x',y')$ in the boundary of $B_{s_n}^{\mathcal F}(x,y)$, it follows from (ii-a), (ii-b) that there is $a\in G_{n}(y)\setminus G_{n-1}(y)$ such that $y'=f_a(y)$, and one can find $x_0, x_1\in \partial$ such that 
\begin{align*}
s_n&=d_{(x,y)}
((x',y'),(x,y))\\
& = 
 d_{(x,y)}((x',y'), (x_0,y')) +  d_{(x,y)}((x_0,y'),(x_1,y)) +  d_{(x,y)}((x_1,y),(x,y)).
\end{align*}
For such a $x_1$, we have
\begin{align}\label{eq:0407d2}
d_{(x,y)}((x_1,y),(x,y)) \le s_n - m_0 -(n-2)R =R_1-m_0 <R.
\end{align}
Hence, we introduce
\begin{align*}
K:= \vert X\vert ^{-1}\inf_{x\in \partial} \left\vert  B_{R_1-m_0}^X(x)\right\vert ,
\end{align*}
which is strictly smaller than $1$ due to the strict inequality of \eqref{eq:0407d2}, where $B_r^X(x)$ is the ball in $X$ of radius $r$ at $x$.

Let $\widetilde{\varphi}$ be a continuous non-negative function on $X $ such that 
\begin{equation}\label{eq:0304c}
\widetilde{\varphi} (x)=0 \quad \text{if  $x\not\in V$}
\end{equation}
and
\begin{equation}\label{eq:0304d}
\int _{X } \widetilde{\varphi} (x) dx= \vert X\vert .
\end{equation}
 Define a continuous function $\varphi : M\to \mathbb R$ by 
\begin{equation}\label{eq:1026p}
\varphi(x,y) = \widetilde{\varphi}(x) \quad \text{ for $(x,y)\in X \times Y$}.
\end{equation}

Fix $x\in V$ and $y\in Y$. Assume that $\lambda _G(y)>0$.
For each $n\in \mathbb N$, it follows from (i-a), (i-b), (ii-a), (ii-b), \eqref{eq:0304c} and \eqref{eq:0304d} that
\begin{align*}
 \int _{B_{r_n}^{\mathcal F}(x,y)} \varphi (w) dw = \int _{B_{s_n }^{\mathcal F}(x,y)} \varphi (w) dw  =\sum _{a \in G_n(y)} \vert X\vert
=\vert X\vert \left\vert G_{n}(y)   \right\vert .
 \end{align*}
On the other hand, it follows from (i-a) that
\[
\vert B_{r_n}^{\mathcal F}(x,y) \vert \ge  \sum _{a \in G_{n}(y) } \vert X\vert =   \vert X\vert \vert G_{n}(y)\vert .
\]
Furthermore, 
it follows from (ii-a), (ii-b), and \eqref{eq:0407d2} that
\begin{align*}
\vert B_{s_n}^{\mathcal F}(x,y) \vert 
&\le \sum _{a \in G_{n-1}(y) } \vert X\vert + \sum _{a \in G_n(y)\setminus G_{n-1}(y) }   K\vert X\vert\\
&= \vert X\vert\left\vert G_{n}(y)     \right\vert - (1 -K)\vert X\vert   \left\vert  G_{n}(y)  \setminus G_{n-1}(y)\right\vert.
\end{align*}
Therefore,
\begin{align*}
 \limsup _{r\to\infty} \frac{1}{\vert B_{r}^{\mathcal F}(x,y) \vert } \int _{B_{r}^{\mathcal F}(x,y)} \varphi (w) dw 
 \le\limsup _{n\to\infty} \frac{1}{\vert B_{r_n}^{\mathcal F}(x,y) \vert } \int _{B_{r_n}^{\mathcal F}(x,y)} \varphi (w) dw  \le 1,
\end{align*}
and by fixing an increasing sequence $\{ n_k\}_{k\ge 1}$ such that $\lim _{k\to \infty}\frac{ \left\vert  G_{n_k}(y)  \setminus G_{n_k-1}(y)\right\vert}{ \left\vert  G_{n_k}(y)   \right\vert} =\lambda _G(y)$, we get
\begin{align*}
 \liminf _{r\to\infty} \frac{1}{\vert B_{r}^{\mathcal F}(x,y) \vert } \int _{B_{r}^{\mathcal F}(x,y)} \varphi (w) dw &\ge \liminf _{k\to\infty} \frac{1}{\vert B_{s_{n_k}}^{\mathcal F}(x,y) \vert } \int _{B_{s_{n_k}}^{\mathcal F}(x,y)} \varphi (w) dw \\
& \ge \liminf _{k\to\infty} \left( 1 - (1 -K)\frac{ \left\vert  G_{n_k}(y)  \setminus G_{n_k-1}(y)\right\vert}{ \left\vert  G_{n_k}(y)   \right\vert} \right)^{-1}\\
& =\left(  \limsup _{k\to\infty} \left( 1 - (1 -K) \frac{ \left\vert  G_{n_k}(y)  \setminus G_{n_k-1}(y)\right\vert}{ \left\vert  G_{n_k}(y)   \right\vert} \right)\right)^{-1}\\
& =\frac{1}{  1 - (1 -K) \lambda_G(y)}.
\end{align*}
Finally, notice that
\[
1< \frac{1}{1 -(1- K)  \lambda _G (y)}
\]
because of the big boundary property $\lambda _G(y) >0$.
These estimates complete the proof of Theorem \ref{mainthm:4}~(2).

\section{Proof of Theorem \ref{mainthm:2} and Corollary \ref{cor:y}}\label{s:6}

\subsection{A pair of pants}
In this subsection,
 we construct a pair of pants with a good Riemannian metric.
This is essential in the construction of a model of leaves of the foliation of the theorem, which will be done in the next subsection.

For a smooth manifold $M$ and a Riemannian metric $g$
 with possible singularities of cone type,
 we write $\|v\|_g$ for the length of a tangent vector $v$
 in the tangent bundle $TM$
 and $|\gamma|_g$ for the length of a piecewise smooth curve $\gamma$ on $M$.
We also write $g_1 \leq g_2$ for two metrics $g_1$ and $g_2$ on $M$
 if $\|v\|_{g_1} \leq \|v\|_{g_2}$ for any $v \in TM$.
In this case, we have
 $|\gamma|_{g_1} \leq |\gamma|_{g_2}$ for any piecewise smooth curve $\gamma$.

First, we define a building block $P$ of our construction.
We define a closed subset $D$ of $\RR^2$ by
\begin{align*}
 D & =([-2,2] \times (-\infty,0]) \cup
  \left(\left[-\frac{3}{2},-\frac{1}{2} \right] \times [0,+\infty)\right)
  \cup
  \left(\left[\frac{1}{2},\frac{3}{2} \right] \times [0,+\infty)\right)
\end{align*}
Let $\sim$ be an equivalence relation on $D$ generated
 by the following relations:
\begin{itemize}
 \item $(-2,y) \sim (2,y)$ for any $y \leq 0$.
 \item $(-\frac{3}{2},y) \sim (-\frac{1}{2},y)$
 and $(\frac{1}{2},y) \sim (\frac{3}{2},y)$ for $y \geq 0$.
 \item $(x,0) \sim (-2-x,0)$ for $x \in [-\frac{1}{2},0]$
 and  $(x,0) \sim (2-x,0)$ for $x \in [0,\frac{1}{2}]$
\end{itemize}
We denote the quotient space $D/{\sim}$ by $P$
 and write $\pi$ for the quotient map from $D$ to $P$.
The space $P$ is an open pair of pants.
See Figure \ref{fig:pants}.
\begin{figure}[ht]
\begin{center}
\includegraphics[scale=0.7]{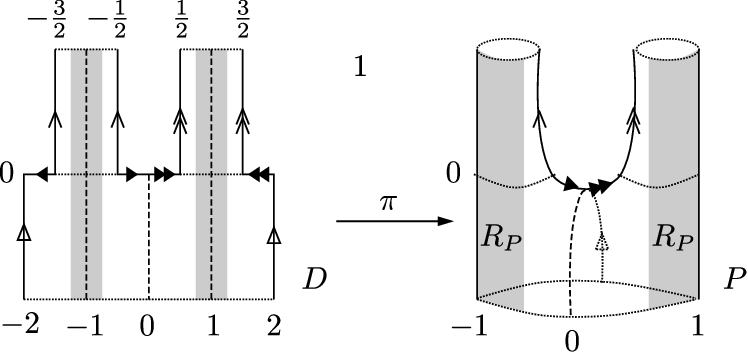}
\end{center}
\caption{A pair of pants $P$}
\label{fig:pants}
\end{figure}
Fix a smooth function $\rho$ on $\RR$ such that
 $\frac{1}{4} \leq \rho(y) \leq 1$ for any $y \in \RR$,
 $\rho(y) = 1$ if $y \geq -\frac{1}{2}$,
 and $\rho(y)=\frac{1}{4}$ if $y \leq -1$.
We define a Riemannian metric $g_D$ on $\RR^2$ by
$g_D=\rho(y)^2 dx^2+dy^2$.
This metric is invariant under the horizontal translation
 $(x,y) \mapsto (x+a,y)$ on $\RR^2$ for any $a \in \RR$ 
 and the rotation $(x,y) \mapsto (-x,-y)$
 on $\RR \times (-\frac{1}{2},\frac{1}{2})$.
This implies that the metric $g_D$ induces a Riemannian metric $g^*_P$
 on $P$ with two singularities at
$\pi(-\frac{1}{2},0)=\pi(-\frac{3}{2},0)$ and
 $\pi(\frac{1}{2},0)=\pi(\frac{3}{2},0)$.
For $r>0$, we put
\begin{equation*}
 Z(r)=\bigcup_{a=\pm \frac{1}{2},\pm\frac{3}{2}} \pi\left(
  \left\{(x,y) \in D \mid (x-a)^2+y^2\leq r^2 \right\}
 \right).
\end{equation*}
By using the partition of unity, we smoothen the cone singularities
 and take a smooth Riemannian metric $\tilde{g}_P$ on $P$
 such that $\tilde{g}_P=g_P^*$ in $P \setminus Z(\frac{1}{10})$.
 By the compactness of $Z(r)$, we may assume that
 there exist a non-negative smooth function $\chi$ on $P$
 and a positive number $\Delta$ such that
 the support of $\chi$ is contained in $Z(\frac{1}{9})$
 and $g_P^* \leq (1+\chi)\tilde{g}_P \leq (1+\Delta) g_P^*$.
We define the smooth metric $g_P$ on $P$ by $g_P=(1+\chi)\tilde{g}_P$, so that
\begin{equation*}
 g_P^* \leq g_P \leq (1+\Delta) g_P^*.
\end{equation*}
Let $d_P$ be the distance on $P$ induced by the metric $g_P$.
We define a function $h_P$ on $P$ by 
\begin{equation}\label{eq:0819b}
h_P(\pi(x,y))=y.
\end{equation}
It is easy to see that $h$ is well-defined and continuous.
\begin{lemma}
\label{lemma:length P} 
$d_P(p,q) \geq |h_P(p)-h_P(q)|$ for any $p,q \in P$.
\end{lemma}
\begin{proof}
Take $(x_0,y_0),(x_1,y_1) \in D$ such that
 $\pi(x_0,y_0)=p$ and $\pi(x_1,y_1)=q$.
Let $\gamma:[0,1] \ra P$
 be the piecewise smooth curve connecting $p$ and $q$.
Put $\gamma(t)=\pi(x(t),y(t))$. 
Then, $y(t)$ is uniquely determined and piecewise smooth with respect to $t$.
Since $g_P \geq g_P^*=\rho(y)^2dx^2+dy^2$, we have
\begin{equation*}
 |\gamma|_{g_P} \geq |\gamma|_{g_P^*} \geq
 \int_0^1 |y'(t)|dt \geq |y(1)-y(0)|= |y_1-y_0|=|h_P(p)-h_P(q)|.
\end{equation*}
Therefore, $d_P(p,q) \geq |h_P(p)-h_P(q)|$.
\end{proof}

We put
\begin{align*}
R_P & = 
 \pi\left(\left[-\frac{11}{8},-\frac{5}{8}\right]\times \RR \right)
 \cup 
  \pi\left(\left[\frac{5}{8},\frac{11}{8}\right]\times \RR \right)
\end{align*}
(recall Figure \ref{fig:pants}).
Remark that $g_P=g_P^*$ on $R_P$ 
 since $R_P$ is disjoint from $Z(\frac{1}{9})$.
This implies that for any $(x,y) \in D$ 
 with $\pi(x,y) \in R_P$ and any $y_0<y$,
 the length of the vertical curve connecting $\pi(x,y)$ and $\pi(x,y_0)$
 with respect to $g_P$ is equal to
 $y-y_0=h_P(\pi(x,y))-h_P(\pi(x,y_0))$.
By Lemma \ref{lemma:length P},
 this vertical curve
 is the shortest curve connecting $\pi(x,y)$ and $\pi(x,y_0)$,
 and hence,
\begin{equation*}
 d_P(\pi(x,y),\pi(x,y_0))=h_P(\pi(x,y))-h_P(\pi(x,y_0)).
\end{equation*}

\subsection{A non-compact surface}\label{s:5.2}
In this subsection,
 we construct a non-compact surface $\Sigma$
 with good metrical properties by gluing copies of
 a compact part of the pair of pants $P$
 described in the previous subsection.

For $a>1$, we put 
\[
D^a=\{(x,y) \in D : |y| \leq a\}, \quad P^a=\pi(D^a).
\]
The space $P^a$ is a compact pair of pants
 with three boundary components
\begin{align}\label{eq:0819d}
\begin{aligned}
 \udel_0 P^a&:=\pi\left(\left[-\frac{3}{2},-\frac{1}{2}\right] \times \{a\}\right), \quad
 \udel_1 P^a:=\pi\left(\left[\frac{1}{2},\frac{3}{2}\right] \times \{a\}\right), \\
\ldel P^a&:= \pi\left([-2,2] \times \{-a\}\right).
\end{aligned}
\end{align}

Fix $L>5(2+\Delta)$ and put
\begin{align*}
 U^{+,0} & =
  \pi\left(\left[-\frac{3}{2},-\frac{1}{2}\right] \times [1,2L-1]\right), &
 U^{+,1} & =
  \pi\left(\left[\frac{1}{2},\frac{3}{2}\right] \times [1,2L-1]\right), \\
 U^- & = \pi([-2,2] \times [-2L+1,-1]).
\end{align*}
See Figure \ref{fig:glue} and notice that
 $P^{2L-1}$ is the disjoint union of $\Int P^1$, $U^{+,0}$, $U^{+,1}$, and $U^-$.
We define diffeomorphisms 
 $\psi^{+,0}:U^{+,0} \ra U^-$ and $\psi^{+,1}:U^{+,1} \ra U^-$ by
\begin{align*}
 \psi^{+,0}(\pi(x,y)) & = \pi(4(x+1),y-2L), &
 \psi^{+,1}(\pi(x,y)) & = \pi(4(x-1),y-2L)    
\end{align*}
 for $\pi(x,y) \in U^{+,0}$ or $U^{+,1}$.
Remark that $\psi^{+,\sigma}$ preserves the metric $g_P$.

Put
\begin{align}\label{eq:0820b}
 \Lambda & =\{(k,\ell) : k\in \ZZ \setminus \{0\},
 \ell=0,1,\dots,2^{|k|-1}-1\}.
\end{align}
Let $\{P_{k,\ell}\}_{(k,\ell) \in \Lambda}$
 be the family of copies of $P$ indexed by $\Lambda$
 and $\vphi_{k,\ell}:P \ra P_{k,\ell}$ the identification map
 for each $(k,\ell) \in \Lambda$.
Define $\pi_{k,\ell}:D \ra P_{k,\ell}$ by
 $\pi_{k,\ell}=\vphi_{k,\ell} \circ \pi$ and put
\begin{align*}
 P^a_{k,\ell} & =\vphi_{k,\ell}(P^a), &
 U^{+,\sigma}_{k,\ell} & =\vphi_{k,\ell}(U^{+,\sigma}), &
 U^-_{k,\ell} & =\vphi_{k,\ell}(U^-). &
\end{align*}
For $(k,\ell)\in \Lambda$, $\sigma \in \{0,1\}$,
 we define a diffeomorphism
 $\psi_{\sigma,(k,\ell)}: U^{+,\sigma}_{k,\ell} \ra U^-_{k',\ell '}$ by
\begin{equation*}
 \psi_{\sigma,(k,\ell)}=\vphi_{k',\ell'}
 \circ \psi^{+,\sigma}
 \circ \vphi_{k,\ell}^{-1},
\end{equation*}
 where 
$ (k',\ell') \in \Lambda$
 with $k'=k +\sgn(k)$ ($\sgn(k)=1$ if $k \geq 1$ and $\sgn(k)=-1$ if $k \leq -1$) and $\ell'=2\ell+\sigma$.
We also define $\psi_0:U^-_{1,0} \ra U^-_{-1,0}$ by
\begin{equation*}
 \psi_0(\pi_{1,0}(x,y))=\pi_{-1,0}(-x,-y-2L).
\end{equation*}
We glue $P^{2L-1}_{k,\ell}$ with $P^{2L-1}_{k',\ell'}$
 by $\psi_{\sigma,(k,\ell)}$
 for all $(k,\ell),(k',\ell') \in \Sigma$
 and $\sigma=0,1$ with $k'=k+\sgn(k)$ and $\ell'=2\ell+\sigma$,
 and glue $P^{2L-1}_{1,0}$ and $P^{2L-1}_{-1,0}$ by $\psi_0$.
Then, we obtain a non-compact surface $\Sigma$ without boundary.
See Figure \ref{fig:glue}.
\begin{figure}[ht]
\begin{center}
\includegraphics[scale=0.7]{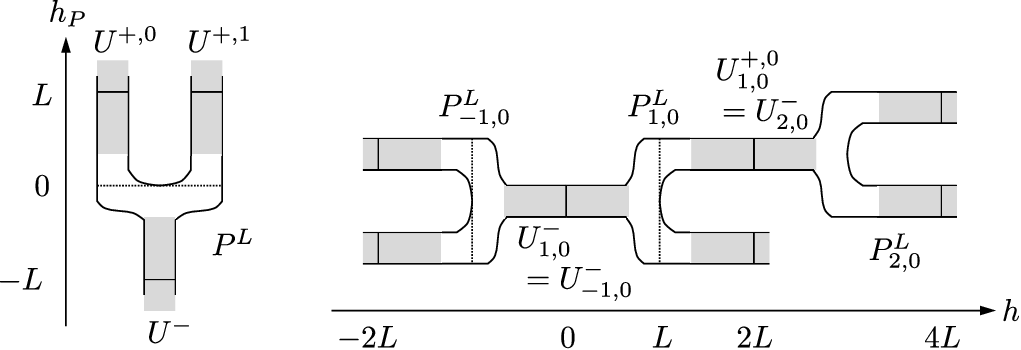}
\end{center}
\caption{The pair of pants $P^L$ and the non-compact surface $\Sigma$}
\label{fig:glue}
\end{figure}
The surface $\Sigma$ can be regarded as
 the union of compact pairs of pants $P^L_{k,\ell}$,
 where the upper boundary $\udel_\sigma P^L_{k,\ell}$ of $P^L_{k,\ell}$
 is glued to
 the lower boundary $\ldel_\sigma P^L_{k',\ell'}$ of $P^L_{k',\ell'}$
 by the gluing map $\psi_{\sigma,(k,\ell)}$
 for all $(k,\ell),(k',\ell') \in \Sigma$ and $\sigma=0,1$  
 with $k'=k+\sgn(k)$ and $\ell'=2\ell+\sigma$,
 and the lower boundaries $\ldel P^L_{1,0}$ and $\ldel P^L_{-1,0}$
 of $P^L_{1,0}$ and $P^L_{-1,0}$ are glued by the gluing map $\psi_0$.
The smooth metric $(\vphi_{k,\ell})_*g_P$
 is preserved by $\psi_{\sigma,(k,\ell)}$
 since $\psi^{+,\sigma}$ preserves the metric $g_P$
 on $P^{2L-1}$.
Hence, the family $((\vphi_{k,\ell})_*g_P)_{(k,\ell) \in \Lambda}$
 of the metric induces a smooth metric $g_\Sigma$ on $\Sigma$.
Let $d_\Sigma$ be the distance on $\Sigma$ induced by the metric $g_\Sigma$.

Put $\Lambda_+=\{(k,\ell) \in \Lambda \mid k \geq 1\}$,
 $\Lambda_-=\{(k,\ell) \in \Lambda \mid k \leq -1\}$,
 and
\begin{align*}
 \Sigma_+ & =\bigcup_{(k,\ell) \in \Lambda_+} P^L_{k,\ell}, &
 \Sigma_- & =\bigcup_{(k,\ell) \in \Lambda_-} P^L_{k,\ell}.
\end{align*}
Note that $ \Sigma_+  \cap  \Sigma_- = \{ \pi _{1,0}(x,-L)  : x\in [-2,2]\} =\{ \pi _{-1,0}(x,L) : x\in [-2,2]\}$. 
We define an involution $I_\Sigma:\Sigma \ra \Sigma$ by
$I_\Sigma(\vphi_{k,\ell}(p))=\vphi_{-k,\ell}(p)$
 for any $(k,\ell) \in \Lambda$ and any $p \in P^L$.
This is an isometry with respect to the metric $g_\Sigma$
 such that $I_\Sigma(\Sigma_+) =\Sigma_-$ and  $I_\Sigma(\Sigma_-) =\Sigma_+$.
We also define $h: \bigcup _{(k,\ell )\in \Lambda } P^{2L-1}_{k,\ell} \to \mathbb R$ by
\begin{equation*}
 h(p)=
\begin{cases}
 (2k-1)L+h_P(\vphi_{k,\ell}^{-1}(p)) & (k \geq 1),\\
 -h(I_\Sigma(p)) & (k \leq -1)
\end{cases}
\end{equation*}
 for $p \in P^{2L-1}_{k,\ell}$. 
 Recall \eqref{eq:0819b} for $h_P$, so that we have
\begin{equation*}
 h(\pi_{k,\ell}(x,y)) =
\begin{cases}
  (2k-1)L+y & (k \geq 1)\\
  (2k+1)L-y & (k \leq -1)
\end{cases}
\end{equation*}
 for any $(k,\ell) \in \Lambda$ and $(x,y) \in D^{2L-1}$.
In particular, 
the map $h: \Sigma \to \mathbb R$ is continuous and satisfies
\begin{align*}
 h(P^L_{k,\ell}) & =
\begin{cases}
 [(2k-2)L,2kL]  & (k \geq 1), \\
 [2kL,(2k+2)L]  & (k \leq -1).
\end{cases} 
\end{align*}
That is, the value of $h(p)$ on $\Sigma$ plays the role of the ``height'' of $p$ (see Figure \ref{fig:glue}).
It holds that
\begin{align}\label{eq:0819c}
 \Sigma_+ =\{p \in \Sigma : h(p) \ge 0\} , \quad \Sigma_- =\{p \in \Sigma : h(p) \le 0\}. 
\end{align}
Note also that 
\begin{align*}
 h(U^-_{k,\ell}) & =
\begin{cases}
 [(2k-3)L+1,(2k-1)L-1]  & (k \geq 1), \\
 [(2k+1)L+1,(2k+3)L-1]  & (k \leq -1).
\end{cases}
\end{align*}

The following is an analog of Lemma \ref{lemma:length P}.
\begin{lemma}
\label{lemma:length Sigma} 
 $d_\Sigma(p,q) \geq |h(p)-h(q)|$ for any $p, q \in \Sigma$.
\end{lemma}
\begin{proof}
Take a piecewise smooth curve $\gamma:[0,1] \ra \Sigma$ such that
 $\gamma(0)=p$ and $\gamma(1)=q$.
Since $\{\Int P^{2L-1}_{k,\ell} : (k,\ell) \in \Lambda\}$
 is an open cover of $\Sigma$,
 there exist sequences $0=t_0<t_1<\dots<t_m=1$
 and $(k_0,\ell_0),\dots,(k_{m-1},\ell_{m-1}) \in \Lambda$
 such that $\gamma([t_j,t_{j+1}]) \subset \Int P^{2L-1}_{k_j,\ell_j}$.
By Lemma \ref{lemma:length P}, we have
 $|(\gamma|_{[t_j,t_{j+1}]})|_g \geq |h(\gamma(t_j))-h(\gamma(t_{j+1}))|$.
Hence,
\begin{equation*}
 |\gamma|_g=\sum_{j=0}^{k-1}|(\gamma|_{[t_j,t_{j+1}]})|_g
 \geq |h(\gamma(1))-h(\gamma(0))|=|h(p)-h(q)|.
\qedhere
\end{equation*}
\end{proof}
We put
$R=\bigcup_{(k,\ell) \in \Lambda}\vphi_{k,\ell}(R_P \cap P^L)$.
\begin{lemma}
\label{lemma:geodesic Sigma}
For any $p \in R \cap \Sigma_+$ and any $\eta_0 \in [-L+1,L-1]$
 with $h(p) \geq \eta_0$, 
 there exists $p_0 \in R \cap h^{-1}(\eta_0)$ such that
 $d(p,p_0)=h(p)-h(p_0)=h(p)-\eta_0$.
\end{lemma}
\begin{proof}
Take $(k,\ell) \in \Lambda$ and $(x,y) \in D^L$
 such that $k \geq 1$ and $p = \pi_{k,\ell}(x,y)$.
We prove the lemma by induction of $k$.
By Lemma \ref{lemma:length Sigma},
 we have $d_\Sigma(p,p') \geq |h(p)-h(p')|$ for any $p' \in \Sigma$.
Hence, it is sufficient to show that
 $d_\Sigma(p,p_0) \leq h(p)-h(p_0)$ for some $p_0 \in R \cap h^{-1}(\eta_0)$.

When $k=1$, put $p_0=\pi_{1,0}(x,\eta_0-L)$.
Then, the vertical segment connecting $p$ and $p_0$ is contained in $R$,
 and hence, its length is $y-(\eta_0-L)=h(p)-h(p_0)$
since $h(p_0)=\eta_0$. Hence, the lemma holds if $k=1$.

Suppose $k \geq 2$ and the lemma holds for $k-1$.
Take $\ell'=0,1,\dots,2^{k-2}-1$
 and $\sigma=0,1$ such that $\ell=2\ell'+\sigma$.
Put $q=\vphi_{k,\ell}(x,-L)$.
Since the vertical curve in $P^L_{k,\ell}$
 connecting $p$ and $q$ is contained in $R$
 and its length is $y-(-L)=h(p)-h(q)$.
Hence, $d_\Sigma(p,q) \leq h(p)-h(q)$.
Since $p$ is contained in $R$,
 we have $x \in [-\frac{11}{8},-\frac{5}{8}] \cup [\frac{5}{8},\frac{11}{8}]$.
In particular, $|x| \in (\frac{1}{2},\frac{3}{2})$, and hence,
 $\frac{x}{4}\pm 1 \in 
 [-\frac{11}{8},-\frac{5}{8}] \cup [\frac{5}{8},\frac{11}{8}]$.
The point $q$ in $P^L_{k,\ell}$ is identified with
 $\pi_{k-1,\ell'}(\frac{x}{4}+(2\sigma-1),L)$
 in $P^L_{k-1,\ell'}$ by the gluing map $\psi_{\sigma,(k-1,\ell')}$.
This implies that $q$ is contained in $R \cap P^L_{k-1,\ell'}$.
See Figure \ref{fig:R}.
\begin{figure}[ht]
\begin{center}
\includegraphics[scale=0.7]{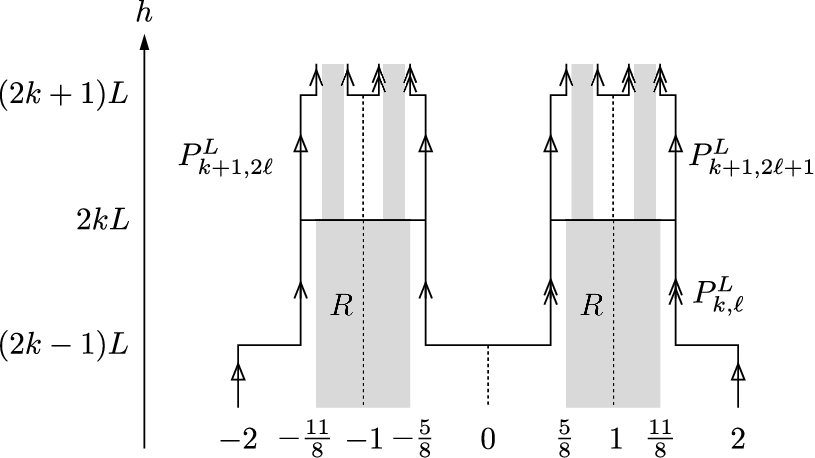}
\end{center}
\caption{Connection between $P_{k+1,2\ell}^L$, $P_{k+1,2\ell+1}^L$ and $P_{k,\ell}^L$}
\label{fig:R}
\end{figure}
By the assumption of induction,
 there exists $p_0 \in R \cap h^{-1}(\eta_0)$
 such that $d_\Sigma(q,p_0) \leq h(q)-h(p_0)$.
Then, we have
\begin{align*}
 d_\Sigma(p,q) \leq d_\Sigma(p,q)+d_\Sigma(q,p_0)\leq h(p)-h(q)+h(q)-h(p_0)=h(p)-h(p_0).
\end{align*}
Therefore, the lemma holds for $k$.
\end{proof}

Recall that the constants $\Delta$ and $L$
 satisfy $g_P \leq (1+\Delta) g_P^*$ and $L>5(2+\Delta)$.
The following is the keystone to control the geometry of a  ball in $\Sigma$.
\begin{prop}
\label{prop:d h}
For $p \in \Sigma_+$ and $p_0 \in U^-_{1,0}$ with $h(p)>h(p_0)$,
 we have
\begin{equation*}
h(p)-h(p_0) \leq d_\Sigma(p,p_0) < h(p)-h(p_0)+ 2+\Delta.
\end{equation*}
\end{prop}
\begin{proof}
By Lemma \ref{lemma:length Sigma},
 we have $h(p)-h(p_0) \leq d_\Sigma(p,p_0)$.
Take $(k,\ell) \in \Lambda$
 and $(x,y) \in D^L$ such that $p=\vphi_{k,\ell}(x,y)$.
There exists
 $x' \in [\frac{1}{8},\frac{7}{8}] \cup [\frac{17}{8},\frac{23}{8}]$
 such that $d_P^*(\pi(x,y),\pi(x',y))<1$, where $d_P^*$ is the distance on $P$ induced by the metric $g_P^*$.
Put $q=\vphi_{k,\ell}(x',y)$.
Then,
\begin{equation*}
 d_\Sigma(p,q) \leq d_P(\pi(x,y),\pi(x',y))
  \leq (1+\Delta) d_P^*(\pi(x,y),\pi(x',y))<1+\Delta. 
\end{equation*}
Since $p_0$ is contained in $U^-_{1,0}$, we have 
$h(p_0) \in [-L+1,L+1]$.
By Lemma \ref{lemma:geodesic Sigma},
 there exists $q_0 \in R$ such that $h(q_0)=h(p_0)$ and $d_\Sigma(q,q_0)=h(q)-h(q_0)$.
Since $|h(p_0)|\leq L-1$, the set $h^{-1}(h(p_0))$ is isometric to $\RR/\ZZ$.
Hence, $d_\Sigma(p_0,q_0)\leq \frac{1}{2}$.
Therefore,
\begin{align*}
 d_\Sigma(p,p_0)  & \leq d_\Sigma(p,q)+d_\Sigma(q,q_0)+d_\Sigma(q_0,p_0)
\leq 1+\Delta+h(q)-h(q_0)+\frac{1}{2}.
\end{align*}
Since $h(p)=h(q)$ and $h(p_0)=h(q_0)$,
 we have $d_\Sigma(p,p_0)<h(p)-h(p_0)+2+\Delta$.
\end{proof}

For any $r>0$ and $p \in \Sigma$,
 let $B(p,r)$ be the closed ball $\{q \in \Sigma : d_\Sigma(p,q) \leq r\}$
 centered at $p$ and of radius $r$.
Put $\Delta_*=2+\Delta$.
Then, $L>5\Delta_*$.
Proposition \ref{prop:d h} and the symmetry by the involution $I_\Sigma$
 imply that 
 for any $p_0 \in h^{-1}([-L+1,L-1])$ and $r>L$,
\[
 \{p \in \Sigma : |h(p)-h(p_0)|\leq r-\Delta_*\}\subset B(p_0,r)
 \subset \{p \in \Sigma : |h(p)-h(p_0)|\leq r\},
\]
in other words,
\begin{equation}
\label{eqn:d h}
h^{-1}\left( B(h(p_0), r-\Delta_* )\right)  \subset B(p_0,r)
 \subset h^{-1}( B(h(p_0), r)).
\end{equation}
\begin{prop}
\label{prop:average 1}
Let $\Phi$ be a continuous function on $\Sigma$ such that
 $\Phi \circ I_\Sigma = -\Phi$,
 $\Phi \geq 0$ on $\Sigma_+$,
 $\Phi=0$ on $h^{-1}([2(k-1)L+1,2kL])$ for any $k\geq 1$, and
 $\int_{P^L_{k,\ell}} \Phi \,d\mathrm{vol}=1$
 for any $(k,\ell) \in \Lambda_+$,
 where $\mathrm{vol}$ is the Riemannian volume given by the metric $g_\Sigma$.
Then,
 for any $p_0 \in h^{-1}([\Delta_*+1,2\Delta_*])$ and any $k \geq 1$,
\[
\frac{1}{\mathrm{vol}(B(p_0,2kL))}\int_{B(p_0,2kL)}\Phi\, d\mathrm{vol}
  \geq \frac{1}{2 \, \mathrm{vol}(P^L)}
\]
and
\[
\int_{B(p_0,2kL-2\Delta_*)}\Phi\, d\mathrm{vol}  = 0.
\]
In particular, the average
 $\frac{1}{\mathrm{vol}(B(p_0,r))}\int_{B(p_0,r)}\Phi\, d\mathrm{vol}$
 of $\Phi$ on $B(p_0,r)$ does not converge as $r$ goes to infinity.
\end{prop}
\begin{proof}
Since $I_\Sigma$ preserves the metric $g$,
 $h \circ I_\Sigma =-h$, and $\Phi \circ I_\Sigma=-\Phi$,
 we have
\begin{align*}
 \int_{h^{-1}([-r,r])} \Phi \,d\vol 
 = \int_{I_\Sigma(h^{-1}[-r,r])}\Phi \circ I_\Sigma \; d\vol
 = -\int_{h^{-1}([-r,r])} \Phi \,d\vol,
\end{align*}
 and hence,
\begin{equation}
\label{eqn:symmetry Phi}
 \int_{h^{-1}([-r,r])} \Phi \,d\vol = 0
\end{equation}
 for any $r >1$.
The assumptions on the support and the integral of $\Phi$ imply that
\begin{align*}
 \int_{h^{-1}([2kL-1,2kL+1])} \Phi \,d\vol
 & =\sum_{\ell=0}^{2^k-1} \int_{P^L_{k+1,\ell}} \Phi \, d\vol = 2^k
\end{align*}
 for any $k \geq 1$.
Equation (\ref{eqn:d h}) implies that
\begin{align*}
B(p_0,2kL) \cap \Sigma_- & \subset h^{-1}([-2kL+1,0]),\\
B(p_0,2kL) \cap \Sigma_+ & \supset h^{-1}([0,2kL+1]).
\end{align*}
Indeed, since $h(p_0)\in [\Delta _*+1, 2\Delta _*]$ and $L>5\Delta_*$, it follows from equations  (\ref{eq:0819c}) and (\ref{eqn:d h}) that
\begin{align*}
B(p_0,2kL) \cap \Sigma_-  &\subset h^{-1}([h(p_0) -2kL ,h(p_0) +2kL]) \cap h^{-1}((-\infty ,0])\\
 &\subset h^{-1}([\Delta _*+1 -2kL ,2\Delta _* +2kL] \cap (-\infty ,0]),
\end{align*}
and
\begin{align*}
B(p_0,2kL) \cap \Sigma_+  &\supset h^{-1}([h(p_0) -2kL+\Delta_* ,h(p_0) +2kL-\Delta_*]) \cap h^{-1}([0,\infty ))\\
 &\supset h^{-1}([3\Delta _* -2kL, 1 +2kL] \cap [0,\infty )) \\
 &\supset h^{-1}([0,1+2kL]).
\end{align*}
Thus, since $\Phi \geq 0$ on $\Sigma_+$ and $\Phi \leq 0$ on $\Sigma_-$,
 we have
\begin{align*}
 \int_{B(p_0,2kL)}\Phi \, d \vol 
 &  \geq \int_{h^{-1}([-2kL+1,2kL+1])} \Phi \,d\vol
=\int_{h^{-1}([2kL-1,2kL+1])} \Phi \, d\vol 
  =2^k,
\end{align*}
 where we use Equation (\ref{eqn:symmetry Phi}) in the first equality.
By Equation (\ref{eqn:d h}) and the fact that $L>5\Delta _*$,
 the disk $B(p_0,2kL)$ is contained in $h^{-1}([-2(k+1)L,2(k+1)L])$.
Hence,
\begin{align*}
 \vol(B(p_0,2kL)) & \leq \vol(h^{-1}([-2(k+1)L,2(k+1)L]))\\
 & =\sum_{(m,\ell) \in \Lambda, |m|\leq k+1}\vol(P^L_{m,\ell})\
  =(2^{k+1}-1)\, \vol(P^L).
\end{align*}
Therefore, 
\begin{align*}
\frac{1}{\vol(B(p_0,2kL))}\int_{B(p_0,2kL)}\Phi\, d\vol
 & \geq \frac{2^k}{2^{k+1}-1} \cdot \frac{1}{\vol(P^L)}
 > \frac{1}{2 \,\vol(P^L)}.
\end{align*}

By Equation (\ref{eqn:d h}), since $h(p_0) \in [\Delta _*+1,2\Delta_*]$, we have
\begin{align*}
B(p_0,2kL-2\Delta_*) & \supset 
 h^{-1}([-2kL+5\Delta_*,2kL-2\Delta_*+1]),\\
B(p_0,2kL-2\Delta_*) 
 & \subset h^{-1}([-2kL+3\Delta_*+1,2kL]).
\end{align*}
Since $\Delta_*=2+\Delta>2$ and $L>5\Delta_*$, this implies that
\begin{align*}
h^{-1}([(-2k+1)L,(2k-1)L]) \subset B(p_0,2kL-2\Delta_*) 
 & \subset h^{-1}([-2kL,2kL]).
\end{align*}
Put 
$A=B(p,2kL-2\Delta_*) \setminus h^{-1}([(-2k+1)L, (2k-1)L])$.
By the assumption on the support of $\Phi$,
 we have $\Phi=0$ on $A$.
By Equation (\ref{eqn:symmetry Phi}), we have
\begin{equation*}
\int_{B(p_0,2kL-2\Delta_*)} \Phi \,d \vol
 =\int_A \Phi \,d \vol =0.
\qedhere 
\end{equation*}
\end{proof}
The above proposition provides a function such that the average on disks does not converge for an open subset.
The following proposition claims the same function admits a point for which the average on disks converges.
\begin{prop}
\label{prop:average 2} 
Let $\Phi$ be a continuous function on $\Sigma$
 such that $\Phi \circ I_\Sigma =-\Phi$.
Then, for any $p_0 \in h^{-1}(0)$ and $r>0$,
\begin{equation*}
 \int_{B(p_0,r)}\Phi \, d\mathrm{vol}=0.
\end{equation*}
\end{prop}
\begin{proof}
Since $I_\Sigma$ preserves the metric $g$, we have, for each $r>0$,
\begin{equation*}
 \int_{B(p_0,r)}\Phi \, d\vol
 = \int_{I_\Sigma(B(p_0,r))}\Phi \circ I_\Sigma \, d\vol
 = \int_{B(p_0,r)} (-\Phi) \,d\vol.
\end{equation*}
This means that $\int_{B(p_0,r)}\Phi \, d\vol=0$.
\end{proof}

\subsection{Construction of the foliated compact 3-manifold}
Here we prove Theorem \ref{mainthm:2}.
Let $X$ be a compact surface with genus three. 
Separate $X$ into a pair of $3$-punctured disks $P_+, P_-$ with the canonical projection $\pi : P_+ \cup P_- \to X$.
Then, $P_+$ and $P_-$ have three common boundary components, denoted by $\partial _AP_+$, $\partial _B P_+$, $\partial _C P_+$ for $P_+$ and by $\partial _AP_-$, $\partial _B P_-$, $\partial _C P_-$ for $P_-$ (that is, for each $\rho \in\{ A, B, C\}$ and $x\in \partial _\rho P_+$, there is exactly one $x'\in \partial _\rho P_-$ such that $\pi (x)= \pi (x')$).

Let $I_A$, $I_B$, $I_C$ be pairwise disjoint open intervals of $S^1$.
For $\sigma \in \{+, -\}$ and $\rho \in \{ A, B, C\}$, let $h_{\rho ,\sigma} : P^L\to P_\sigma$ be a diffeomorphism satisfying that
\begin{equation}\label{eq:0820a}
h_{\rho ,\sigma} (\ldel P^L) = \partial _\rho P_\sigma , \quad h_{\rho ,\sigma } (\udel_s P^L) = \partial _{\rho + s+1} P_{\sigma}
\end{equation}
for $s\in \{0,1\}$, where ``$+1$'' on $\{ A, B, C\}$ is interpreted as the left-shift operation, i.e.~
\[
A+1=B, \quad B+1=C, \quad C+1=A
\]
(recall \eqref{eq:0819d} for $\ldel P^L$ and $\udel_s P^L$).

 \begin{figure}[h]
\begin{center}
\includegraphics[scale=0.75]{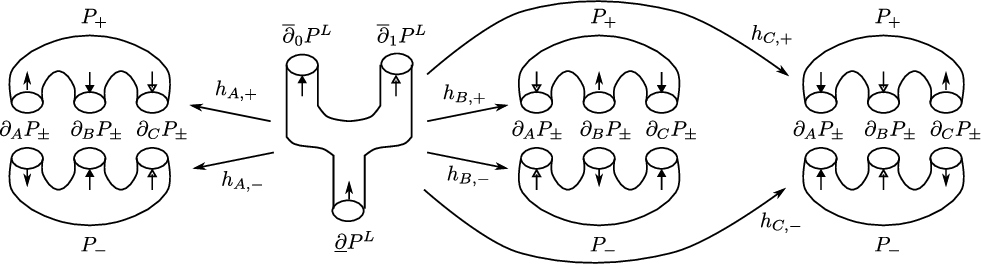}
\end{center}
\caption{The gluing maps $h_{\rho,\sigma}$}
\label{fig:h}
\end{figure}

We further introduce three diffeomorphisms $f_A, f_B, f_C: S^1 \to S^1$ such that 
\begin{itemize}
\item $I_B \cup I_C \subset f_A(I_A)$, $f_A(I_B) \cup f_A(I_C) \subset I_A$,
\item $I_C \cup I_A \subset f_B(I_B)$, $f_B(I_C) \cup f_B(I_A) \subset I_B$,
\item $I_A \cup I_B \subset f_C(I_C)$, $f_C(I_A) \cup f_C(I_B) \subset I_C$,
\end{itemize}
and that there exists an open interval $J\subset I_A$ satisfying that $f_A(J) \subset I_A.$
Notice that  
\[
I_B \cup I_C \subset f_A^{-1}(I_A), \quad f_A^{-1}(I_B) \cup f_A^{-1}(I_C) \subset I_A
\]
 with a similar property for $f_B, f_C$,  and 
 \[
 f_A^{-1}(J) \subset I_A.
 \]

  \begin{figure}[h]
\begin{center}
\includegraphics[scale=1]{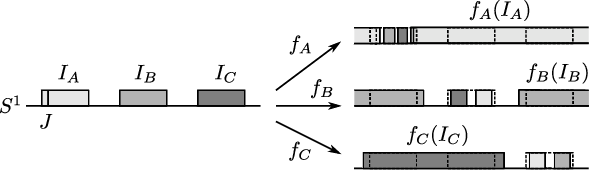}
\end{center}
\caption{The images of $I_\rho$ by $f_\rho$}
\label{fig:f}
\end{figure}

 Let 
 $\{g_x'\}_{x\in P_+ \cup P_-}$ be a smooth family of $\mathcal C^\infty$ metrics on $S^1$
  such that 
  \[
 g_x '= (f_\rho )_*g_{x'}' \quad \text{if $x\in \partial _\rho P_+$, $x'\in \partial _\rho P_-$ and $\pi (x)=\pi (x')$}
 \] 
 for each $\rho \in \{ A, B, C\}$.
Let $g$ be a $\mathcal C^\infty$ metric  on $(P_+ \cup P_-) \times S^1$ satisfying
\[
g_{(x,y)}((v,w),(v',w'))= 
((h_{\rho ,\sigma} )_* g_P)(v,v') + g'_x(w,w') 
\quad \text{if $(x,y) \in P_\sigma \times 
 I_\rho$}
\]
for $v, v'\in T_{x}(P_+ \cup P_-)$ and $w, w'\in T_{y}S^1$ (one can explicitly construct such a metric $g$ from $g_P$ by using bump functions).

Finally, we glue the boundary of $(P_+ \cup P_-)\times S^1$ by the equivalence relation $\sim$ given by
\[
(x,y) \sim (x', f_\rho^{-1}(y)) \quad \text{with $x\in \partial _\rho P_+$, $x'\in  \partial _\rho P_-$, $\pi(x)=\pi (x')$, $\rho \in\{A,B,C\}$},
\]
and denote by $M$ the resulting 3-manifold.
As a usual abuse of notation, we identify $(x,y)\in (P_+ \cup P_-) \times S^1$ with its equivalent class in $M$. 
Observe that this gluing preserves 
 the metric $g$, so we again denote by $g$ the induced metric on $M$.
Let $\mathcal F$ be the foliation of $M$ consisting of leaves of the form $(\bigcup _{g\in G} (P_+ \cup P_-) \times \{ g(y)\})/\sim $ with $y\in S^1$, where $G$ is the group generated by $G_1:=\{f_A^\pm , f_B^\pm , f_C^\pm\}$.

  \begin{figure}[h]
\begin{center}
\includegraphics[scale=1]{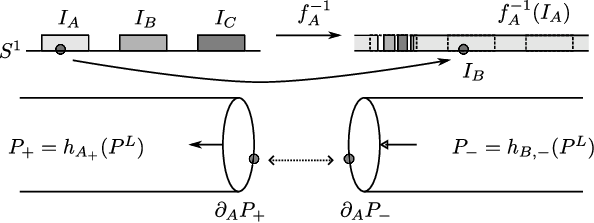}
\end{center}
\caption{Gluing of $\partial _AP_+ \times S^1$ and $\partial _AP_- \times S^1$}
\label{fig:glue2}
\end{figure}

\subsection{Completion of the proof}\label{ss:6.4}

We first see that for each leaf $L\in \mathcal F$ through an open set, the non-compact Riemannian manifold $(L,g\vert _L)$, where $g\vert _L$ is the restriction of $g$ on $L$, is isometrically diffeomorphic to $(\Sigma , g_\Sigma )$ of Section \ref{s:5.2}. 
Fix a point $(x_0, y_0)$ from an open set
\begin{equation}\label{eq:1119b}
h_{A,+} \circ h^{-1}((\Delta _*+1,2\Delta _*))\times J.
\end{equation}
According to \eqref{eq:0820a},
define $\{ f_{k,\ell} \}_{(k,\ell )\in \Lambda }$ as follows (recall \eqref{eq:0820b} for $\Lambda$):
\begin{itemize}
\item  Set $f_{1,0}=\mathrm{id}_{S^1}$ and $f_{-1,0}=f_A^{-1}$;
\item Given $(k,\ell)\in \Lambda$ and $s\in \{0,1\}$, if $f_{k,\ell}(J) \subset I_{\rho }$
 with $\rho \in \{ A, B, C\}$, then set
\[
f_{k+\mathrm{sgn}(k), 2\ell +s } = f_{\rho +s +1}^{\mathrm{sgn}(k)} \circ f_{k,\ell},
\]
where 
\[
\mathrm{sgn} (k)=
\begin{cases} +1 \quad & (k\in \{2k' -1: k'\in \mathbb N\} \cup \{ -2k' : k' \in \mathbb N\})\\
-1  \quad &  (k\in \{2k' : k'\in \mathbb N\} \cup \{ -2k' +1: k' \in \mathbb N\})
\end{cases}
\]
(note that $f_{k+\mathrm{sgn}(k), 2\ell +s }(J) \subset I_{\rho + s +1}$).
\end{itemize}
Moreover, define $\sigma(k)$ as $\sigma(k)=+$ if $\mathrm{sgn} (k) =+1$ and $\sigma(k)=-$ if $\mathrm{sgn} (k) =-1$. 
For example, if $(k,\ell)=(1,0)$, then $\sigma (k)=+$ and $\rho =A$, $\rho +1=B$, $\rho +2=C$, so that $f_{2,0} =f_{B}\circ f_{1,0}$ and  $f_{2,1} =f_{C}\circ f_{1,0}$. 
If $(k,\ell)=(-1,0)$, then $\sigma (k)=-$ and $\rho =A$, $\rho +1=B$, $\rho +2=C$, so that $f_{-2,0} =f_{B}^{-1}\circ f_{-1,0}$ and  $f_{-2,1} =f_{C}^{-1}\circ f_{-1,0}$.
By the above notation, we can write the leaf $B^{\mathcal F}(x_0,y_0)$ through $(x_0,y_0)$ as
\[
\left(\bigcup _{(k,\ell)\in \Lambda } P_{\sigma (k)} \times \{ f_{k,\ell}(y_0)\} \right)/\sim ,
\]
which is, by construction, diffeomorphic to $\Sigma = \bigcup _{(k,\ell )\in \Lambda } P^{L}_{k,\ell} $ by  the diffeomorphisms $\{ \varphi _{k,\ell } \circ h_{\rho (k,\ell ), \sigma (k)}^{-1}\} _{(k,\ell )\in \Lambda }$ with $\rho (k,\ell)$ determined by $f_{k,\ell}(J) \subset I_{\rho (k,\ell)}$, and moreover, the pushforward of $g\vert _L$ by the diffeomorphisms is $g_\Sigma$.

Now we apply Proposition \ref{prop:average 1} to get a continuous function $\Phi :\Sigma \to \mathbb R$ for which the average  $\lim_{r\to\infty}\frac{1}{\mathrm{vol}(B(p_0,r))}\int_{B(p_0,r)}\Phi\, d\mathrm{vol}$ does not exist for $p_0=\varphi _{1,0} \circ h_{A, +}^{-1}(x_0)$.
According to the designed requirement on $\Phi$ in Proposition \ref{prop:average 1}, we can assume that $\Phi \circ \varphi_{k,\ell} =\Phi \circ \varphi _{1,0}$ on $P^L$ for any $(k,\ell)\in \Lambda $.
Note also that $\{ f_{k,\ell}(J)\}_{(k,\ell)\in \Lambda}$ are mutually disjoint open sets included in $I_A \cup I_B \cup I_C$.
Hence, we can define a continuous function $\Psi$ on 
$
\left(\bigcup _{(k,\ell)\in \Lambda } P_{\sigma (k)} \times \{ f_{k,\ell}(J)\} \right)/\sim 
$
 by
\begin{equation}\label{eq:1119c}
\Psi (x,y) = \Phi \circ  \varphi _{k,\ell }\circ h_{\rho (k,\ell),\sigma (k)}^{-1}(x) \quad \text{if $(x,y) \in P_{\sigma (k)} \times \{ f_{k,\ell}(J)\}$}.
\end{equation}
Indeed, this is consistent with the equivalence relation $\sim$. 
If $(x,y) \in \partial _{\rho+s+1}P_{\sigma} \times \{ f_{k,\ell}(J)\}$ for some $s\in \{0,1\}$ with $\rho=\rho(k,\ell)$ and $\sigma =\sigma (k)$,
 then
\[
(x,y) \sim (x',f_{\rho +s+1}^{\mathrm{sgn}(\sigma)}(y))
\]
with $x' \in  \partial _{\rho +s+1}P_{-\sigma}$ satisfying $h_{\rho +s+1,\sigma }^{-1}(x) = h_{\rho +s+1,-\sigma}^{-1}(x')$, where $\mathrm{sgn} (\pm )=\pm 1$.
On the other hand, since $f_{\rho +s +1}^{\mathrm{sgn}(\sigma)} \circ f_{k,\ell} =f_{k+\mathrm{sgn}(k), 2\ell +s}$,
\begin{align*}
\Psi (x',f_{\rho +s+1}^{\mathrm{sgn}(\sigma)}(y))  &  =\Phi \circ \varphi _{k+\mathrm{sgn}(k),2\ell +s}\circ h_{\rho  +s+1,-\sigma}^{-1}(x')\\
& =\Phi \circ \varphi _{k,\ell}\circ h_{\rho +s+1,\sigma}^{-1}(x) 
=\Psi (x,y).
\end{align*}
Here we used the property $\Phi \circ \varphi_{k,\ell} =\Phi \circ \varphi _{1,0}$ for any $(k,\ell)\in \Lambda $.
Finally, we extend $\Psi$ to a continuous function on $M$, which is possible because $\{ f_{k,\ell}(J)\}_{(k,\ell)\in \Lambda}$ are mutually disjoint open sets.

By 
 the above observation that $B^{\mathcal F}(x_0,y_0)=\left(\bigcup _{(k,\ell)\in \Lambda } P_{\sigma (k)} \times \{ f_{k,\ell}(y_0)\} \right)/\sim$ is isometrically diffeomorphic to $\Sigma = \bigcup _{(k,\ell )\in \Lambda } P^{L}_{k,\ell} $ by $\{ \varphi _{k,\ell } \circ h_{\rho (k,\ell ), \sigma (k)}^{-1}\} _{(k,\ell )\in \Lambda }$,  it holds that $\vert B^{\mathcal F}_r(x_0,y_0)\vert  = \mathrm{vol}(B(p_0,r))$ and 
\begin{align*}
 \int _{B^{\mathcal F}_r(x_0,y_0)} \Psi (z) \, dz& = \sum_{(k,\ell)\in \Lambda} \int _{B^{\mathcal F}_r(x_0,y_0) \cap (P_{\sigma (k)} \times \{ f_{k,\ell}(y_0)\})}  \Phi \circ  \varphi _{k,\ell }\circ h_{\rho (k,\ell),\sigma (k)}^{-1}(x) \, dx\\
 &  = \sum_{(k,\ell)\in \Lambda} \int _{B_r(p_0) \cap P_{k,\ell}^L}  \Phi (x) \, d\mathrm{vol}
   =   \int _{B_r(p_0) }  \Phi (x) \, d\mathrm{vol}.
\end{align*}
Therefore, it follows from Proposition \ref{prop:average 1} that the length average of $\Psi$ along the leaf $B^{\mathcal F}_r(x_0,y_0)$ does not exist.
This completes the proof of Theorem \ref{mainthm:2}.

\subsection{Proof of Corollary \ref{cor:y}}
We first show the following general proposition, which would be useful to extend foliations with regular/irregular behavior to higher dimension.
\begin{prop}\label{prop:Tk}
Let $l, l',d,d'$ be integers such that $0<l<d$ and $0\le l' \le d'$.
Let $(T_1,g_1)$, $(T_2,g_2)$ be compact Riemannian manifolds whose dimensions are $d'-l'$ and $l'$, respectively.
Let $\mathcal{F}$ be a codimension $l$ foliation on a $d$-dimensional Riemannian manifold $(M,g)$, and 
consider the codimension $(l+l')$ foliation 
\[
\widetilde{\mathcal{F}} := \left\{ L \times T_1 \times \{ \theta_2\}  \mid L\in \mathcal F, \; \theta_2 \in T_2\right\}
\]
on the $(d+d')$-dimensional Riemannian manifold $\widetilde M:=M \times T_1\times T_2$ equipped with the product Riemannian metric $g\times g_1\times g_2$.
(In the case when $d'=0$, we interpret $\widetilde M$ as $M$.
When $d'> 0$ and in addition $l'=0$ or $d'-l'=0$, we interpret $\widetilde M$ as $M\times T_1$ or $M\times T_2$, respectively.) 
Let $\varphi \colon M \to \R$ be a continuous function, and consider a continuous function $\widetilde{\varphi} \colon \widetilde M  \to \R$ defined by 
\[
\widetilde{\varphi}(x, \theta_1,\theta_2) = \varphi(x) \quad \text{ for $(x, \theta_1,\theta_2) \in M \times T_1\times T_2$}.
\]

Given $x\in M$, suppose that a sequence $\{r_n\}_{n\in\mathbb N}$ with $\lim_{n\to\infty} r_n=\infty$ satisfies 
 \[
B_{r_n}^{\mathcal{F}}(x) \setminus B_{r_n - R_1}^{\mathcal{F}}(x) \subset \varphi^{-1}(0),
 \]
where $R_1$ is the diameter of $T_1$.
  Then, 
 \begin{align*}
 &\limsup_{n\to\infty}\dfrac{1}{\vert B_{r_n}^{\mathcal{F}}(x) \vert} \int_{B_{r_n}^{\mathcal{F}}(x)} \varphi(y) \, dy
 =\limsup_{n\to\infty}\dfrac{1}{\vert B_{r_n}^{\widetilde{\mathcal{F}}}(x,0,0)\vert}\int_{B_{r_n}^{\widetilde{\mathcal{F}}}(x,0,0)} \widetilde{\varphi}(p) \, dp,\\
 & \liminf_{n\to\infty}\dfrac{1}{\vert B_{r_n}^{\mathcal{F}}(x) \vert} \int_{B_{r_n}^{\mathcal{F}}(x)} \varphi(y) \, dy
 =\liminf_{n\to\infty}\dfrac{1}{\vert B_{r_n}^{\widetilde{\mathcal{F}}}(x,0,0)\vert}\int_{B_{r_n}^{\widetilde{\mathcal{F}}}(x,0,0)} \widetilde{\varphi}(p) \, dp.
 \end{align*}
\end{prop}
\begin{proof}
We first note that for any $r\ge R_1$,
\begin{equation}\label{eq:1119a}
B_{r - R_1}^{\mathcal{F}}(x)\times T_1 \times \{0\} \subset   B_{r}^{\widetilde{\mathcal{F}}}(x,0,0) \subset B_{r}^{\mathcal{F}}(x)\times T_1 \times \{0\}. 
\end{equation}
The second inclusion is obvious, and the first inclusion holds because if $(y,\theta_1, 0)\in B_{r - R_1}^{\mathcal{F}}(x)\times T_1 \times \{0\}$, then it follows from the product structure of $g\times g_1\times g_2$ that
\[
d\left((y,\theta_1, 0) , (x,0,0)\right) ^2\le  d(y,x)^2 + d(\theta _1,0) ^2\le (r-R_1)^2 + R_1^2\le r^2.
\]
Consequently, we get
\[
\vert B_{r - R_1}^{\mathcal{F}}(x)\vert \cdot \vert  T_1 \vert \leq \vert B_{r}^{\widetilde{\mathcal{F}}}(x,0,0)\vert \leq \vert B_{r}^{\mathcal{F}}(x)\vert \cdot \vert  T_1 \vert .
\]

Fix $x\in M$ and consider  a sequence $\{r_n\}_{n\in\mathbb N}$ satisfying $\lim_{n\to\infty} r_n=\infty$ and $ B_{r_n}^{\mathcal{F}}(x) \setminus B_{r_n - R_1}^{\mathcal{F}}(x) \subset \varphi^{-1}(0).$
Then, it follows from \eqref{eq:1119a} that 
\[
\vert  T_1 \vert \, \int_{B_{r_n -R_1}^{\mathcal{F}}(x)} \varphi(y) \, dy = \int_{B_{r_n}^{\widetilde{\mathcal{F}}}(x,0,0)} \widetilde{\varphi}(p) \, dp = \vert  T_1 \vert \, \int_{B_{r_n}^{\mathcal{F}}(x)} \varphi(y) \, dy.
\]
Combining it with the above estimate, we get
\[
\dfrac{\int_{B_{r_n}^{\mathcal{F}}(x)} \varphi(y) \, dy}{\vert B_{r_n}^{\mathcal{F}}(x) \vert} 
 \leq \dfrac{\int_{B_{r_n}^{\widetilde{\mathcal{F}}}(x,0,0)} \widetilde{\varphi}(p) \, dp}{\vert B_{r_n}^{\widetilde{\mathcal{F}}}(x,0,0)\vert}
  \leq \dfrac{ \int_{B_{r_n -R_1}^{\mathcal{F}}(x)} \varphi(y) \, dy}{\vert B_{r_n - R_1}^{\mathcal{F}}(x)\vert}.
\]
Therefore, the claim immediately follows.
\end{proof}

The following is an immediate consequence of Proposition \ref{prop:Tk}. Notice that for arbitrary  $R_1>0$, 
one can find a compact Riemannian manifold $(T_1,g_1)$ (of arbitrary dimension)
 whose diameter is $R_1$.
\begin{corollary}\label{cor:Tk}
Let $M, \mathcal F, \varphi$ be as in Proposition \ref{prop:Tk}. Assume that  there are an open set $U\subset M$, $R_1>0$ and sequences $\{r_n^+ \}_{n\in\mathbb N}$, $\{r_n^-\}_{n\in\mathbb N}$ with $\lim_{n\to\infty} r_n^\pm =\infty$ such that for any $x\in U$, it holds that $B_{r_n^\pm}^{\mathcal{F}}(x) \setminus B_{r_n^\pm - R_1}^{\mathcal{F}}(x) \subset \varphi^{-1}(0)$ and
 \[
\lim_{n\to\infty}\dfrac{1}{\vert B_{r_n^+}^{\mathcal{F}}(x)\vert } \int_{B_{r_n^+}^{\mathcal{F}}(x)} \varphi(y) \, dy > \lim_{n\to\infty}\dfrac{1}{\vert B_{r_n^-}^{\mathcal{F}}(x)\vert } \int_{B_{r_n^-}^{\mathcal{F}}(x)} \varphi(y) \, dy.
\]
Then, one can find extensions $ \widetilde M, \widetilde{\mathcal F}, \widetilde \varphi $ in the sense of Proposition \ref{prop:Tk} such that,
for any $(x,\theta _1, \theta _2)$ in the open set $U\times T_1\times T_2$,
the length average of $\widetilde{\varphi}$ at $(x,\theta_1,\theta_2) $ does not exist.
\end{corollary}

\begin{proof}[Proof of Corollary \ref{cor:y}]
In order to prove Corollary \ref{cor:y}, we apply Corollary \ref{cor:Tk} to the codimension one foliation $\mathcal F$ on the 3-manifold $M$ given in Theorem \ref{mainthm:2} (constructed in the previous subsections of this section), so that, with $\mathbf d:=3+d'$ and $\mathbf l:=1+l'$, the resulting manifold $\widetilde M$ (obtained in the manner of Corollary \ref{cor:Tk}) is $\mathbf d$-dimensional and 
the resulting foliation $\widetilde{\mathcal F}$ is condimension $\mathbf l$.
By $0\le l'\le d'$, we have that $\mathbf d\ge 3$ and $1\le \mathbf l \le 1+d' = \mathbf d-2$, as required in Corollary \ref{cor:y}.
Therefore, it suffices to check that 
\begin{equation}\label{eq:1119d}
B_{r_n^\pm}^{\mathcal{F}}(x) \setminus B_{r_n^\pm - R_1}^{\mathcal{F}}(x) \subset \varphi^{-1}(0) \quad \text{ for any $x\in U$} 
\end{equation}
with
\begin{itemize}
\item $U=h_{A,+} \circ h^{-1}((\Delta _*+1,2\Delta _*))\times J$ (see \eqref{eq:1119b});
\item $r_n^+=2kL$ and $r_n^-=2kL-2\Delta _*$ (see Proposition \ref{prop:average 1});
\item $\varphi = \Psi$ (see \eqref{eq:1119c}).
\end{itemize} 
On the other hand, since $\Phi=0$ on $h^{-1}([2(k-1)L+1,2kL])$ and $5\Delta _* <L$ (see Proposition \ref{prop:average 1}), the required property \eqref{eq:1119d} holds for any $R_1 <\Delta _*$. 
  This completes the proof of Corollary \ref{cor:y}.
\end{proof}

\subsection{Proof of Proposition \ref{prop:2.6}}
We first note that any point $x\in \mathbb T^d$ is not periodic, that is, $f(t)(x)\neq x$ for any $t\in \mathbb R^l\setminus\{0\}$ by the hypothesis of Proposition \ref{prop:2.6} because $f(t)(x)=x+t\alpha \mod \mathbb Z^d$ with $t\alpha:=\sum_{j=1}^l t_j\alpha^{(j)}$.
By construction, this immediately implies the first assertion, so it suffices to show that $\beta_r\delta_x:=\frac{1}{\vert B_r\vert}\int_{B_r}\delta_{f(t)(x)}dt$ weakly converges to the Lebesgue measure for every $x\in\mathbb T^d$.

Next, notice that the Lebesgue measure is an invariant measure of $f$ by the translation invariant property of the Lebesgue measure, and any measure except the Lebesgue measure cannot be an invariant measure of $f$ due to the fact that some $\alpha ^{(j)}$ is irrational and the well-known fact that the Lebesgue measure is a unique invariant probability measure for any irrational flow (cf.~\cite[Propositions 4.2.2 and 4.2.3]{KH1995}). Hence, the Lebesgue measure is the unique invariant probability measure of $f$.

By the compactness of the space of all Borel probability measures on a compact metric space with respect to the weak topology, for any $x\in\mathbb T^d$, 
$\{\beta_r\delta _x\}_{r>0}$ has a convergent subsequence $\{\beta_{r_n}\delta _x\}_{n\in\mathbb N}$.
Denoting the limit of the subsequence by $\mu_x$, we get that for any $t\in\mathbb R^l$,
\[
\mu _x(f(t)^{-1}(A))=\lim_{n\to\infty}(\beta_{r_n}\delta _{f(t)(x)})(A)
=\lim_{n\to\infty}\frac{1}{\vert B_{r_n}\vert}\int_{B_{r_n}}\delta_{f(s)\circ f(t)(x)}(A)ds.
\]
Note that $f(s)\circ f(t)=f(s+t)$, thus with notations $a_s:=\delta_{f(s)(x)}(A)$ and $B_{r_n}(t):=\{s+t : s\in B_{r_n}\}$, 
\[
\int_{B_{r_n}}\delta_{f(s)\circ f(t)(x)}(A)ds
=
\int_{B_{r_n}}a_s\; ds
+
\int_{B_{r_n}(t)\setminus B_{r_n}}a_s\; ds
-
\int_{B_{r_n}\setminus B_{r_n}(t)}a_s\; ds,
\]
while
\[
\left\vert\frac{1}{\vert B_{r_n}\vert}\int_{B_{r_n}(t)\setminus B_{r_n}}a_s\; ds\right\vert \le \frac{\vert B_{r_n}(t)\setminus B_{r_n}\vert }{\vert B_{r_n}\vert}\to 0 \quad \text{as $r_n\to \infty$},
\]
and similarly $\frac{1}{\vert B_{r_n}\vert}\int_{B_{r_n}\setminus B_{r_n}(t)}a_s\; ds \to 0$.
Hence,
\[
\mu _x(f(t)^{-1}(A))=\lim_{n\to\infty}\frac{1}{\vert B_{r_n}\vert}\int_{B_{r_n}}\delta_{f(s)(x)}(A)ds=\mu _x(A).
\]
Namely, $\mu _x$ is an $f$-invariant probability measure.
Since the Lebesgue measure is a unique invariant probability measure, $\mu_x$ is the Lebesgue measure, and this completes the proof of Proposition \ref{prop:2.6}.

\section*{Acknowledgement}
YN and TY are grateful to the members of University of Porto for their warm hospitality.
PV acknowledges financial support from Funda\c c\~ao para a Ci\^encia e Tecnologia (FCT) - Portugal through the grant CEECIND/03721/2017 of the Stimulus of Scientific Employment, Individual Support 2017 Call and by CMUP, under the project UIDB/00144/2020. 
This work was also partially supported by JSPS KAKENHI Grant  Numbers 19K14575, 20K03583, 22K03302, 23K03188, and 24K06733.

\bibliographystyle{abbrv}
\bibliography{ANYV}

\end{document}